\documentclass[a4paper,man,natbib,floatsintext]{apa6}

\usepackage{graphicx ,multicol,lipsum ,amssymb ,amsfonts ,color ,amsmath ,amsthm ,tikz,enumitem ,forest ,lscape ,relsize ,float ,pdfpages ,fancyhdr ,rotating ,array ,fancyhdr,kbordermatrix,blkarray,multirow,babel,inputenc,todonotes, censor,graphicx,multicol,lipsum,amssymb,amsfonts,color,graphicx,soul,amsmath,tikz,enumitem,forest,tikz-qtree,xcolor}
\usepackage{afterpage}
\usepackage[hyphens]{url}

\setstcolor{red}

\newcommand\tab[1][1cm]{\hspace*{#1}}

\usetikzlibrary{shapes,arrows}

\tikzstyle{cloud} = [draw, ellipse,fill=blue!20, node distance=2.5cm,minimum height=2em]
\tikzstyle{io} = [trapezium, trapezium left angle=70, trapezium right angle=110,minimum height=2em,text centered, draw=black, fill=blue!30]
\tikzstyle{decision} = [diamond, draw, fill=red!20, text width=5em, text badly centered, node distance=3cm, inner sep=5pt]
\tikzstyle{block} = [rectangle, draw, fill=blue!20, text width=5em, text centered, rounded corners, minimum height=4em]
 \tikzstyle{line} = [draw, -latex']

\newtheorem{theorem}{Theorem}
\newtheorem*{theorem*}{Theorem}

\newtheorem{definition}{Definition}
\newtheorem*{definition*}{Definition}

\newtheorem*{exmp*}{Example}

\title{Evaluating and Ranking Criteria for Mathematics Graduate Education Admissions through the Analytic Hierarchy Process}
\shorttitle{Ranking Graduate Admissions Criteria}
\author{Simon D. Nguyen}
\affiliation{University of California, Irvine}

\abstract{Graduate school admission committees consider many factors for admission into a mathematics PhD program, and aspiring applicants often wonder which factors are more important. Applicants discerning where to best dedicate their time may ask questions such as, "should I take graduate level courses or participate in research?" This paper seeks to answer such questions by constructing an ordinal ranking of admission criteria and provide insight on why some factors are more valued. Using a conjoint analysis method called Analytic Hierarchy Process from mathematician Thomas Saaty, this paper evaluates the relative influence certain predictors have on admissions into a mathematics PhD program. Additionally, this paper analyzes the difference in factor rankings between varying populations. For instance, results indicate that full professors tend to differ from their associate and assistant peers in their criteria rankings, and top 30 math PhD programs also differ in rankings when compared to programs outside the top 30. It is my hope that the ordinal rankings and its subsequent analyses will not only explain why some factors are valued more by the mathematics community but also provide some guidance to  potential graduate school applicants on where to dedicate their time as undergraduates.

\vfill

\textit{Keywords}: Admissions, Analytic Hierarchy Process, Education, Multi-Criteria Decision Analysis}
\begin{document}
\raggedbottom

\maketitle

\tableofcontents
\newpage

\listoffigures
\newpage

\section{Acknowledgements}
My foremost gratitude belongs to Professor Patrick Q. Guidotti for his time as my advisor. His supervision not only provided valuable guidance but more importantly helped developed my academic curiosity as a researcher. As this study was supported by the University of California, Irvine's Undergraduate Research Opportunities Program, I am appreciative for its generous funding. I would also like to thank the many mathematics faculty members who took the time to fill out my surveys, be interviewed, and provide feedback on my paper. I am ever thankful for the many ways they contributed to my study, and I am grateful that I was able to connect with many of them virtually. Finally, I would like to dedicate this paper to my father Norbert.
\newpage

\section{Introduction}
Time is a limited and transient resource, and undergraduate students are often confused on where to best commit their time to prepare themselves for graduate education. In a "short" four years, undergraduates must signal to admission committees that they are (1) academically proficient if not excellent, and (2) have research potential. These two goals are broadly the primary signals applicants seek to convey in an application, but the method in which a student conveys these signals will vary. As a result, students must make strategic decisions on where to dedicate their time.

For instance, is it better for a student to take the harder analysis sequence to signal academic excellence even at the expense of a lower grade? Likewise, should a student dedicate his summer studying for the math subject GRE or attend a full-time REU program? Choosing either options will improve an applicant's admission chances, but students are interested in which option provides the greater marginal benefit. Such questions suggest there is some loose preference of admission criteria, and this paper will seek to provide an ordinal ranking of these criteria.

Additionally, this paper's interest also lies in whether admission criteria preferences will differ between the types of professors and graduate programs. Questions such as "do full professors prefer research experience more than assistant professors, or do the top 30 mathematics PhD programs place less emphasis on standardized testing?" will be approached in this study. This paper, then, will also consider preference differences between professorships and graduate programs.

The structure of this paper begins broadly by guiding the reader through the current literature around graduate admissions but will narrow down to the paper's specific phenomenon of interest: admission criteria preferences. The paper then introduces the theoretical framework and methodology on how to analyze admission preferences followed by the results of the analysis. This paper concludes with a discussion of the findings and their relation to our phenomenon of interest.

\section{Literature Review}
\subsection{On Mathematics Education}
There does not exist much literature on the supply-side on graduate admissions. Rather, current literature is primarily interested on the demand-side of school choice, often using multi-criteria decision analysis (MCDA) to evaluate how an applicant chooses one school above others. However, this paper instead evaluates how schools choose an applicant on the supply-side using MCDA. The following section will give a guided tour through articles that provide background for this paper.

In Timmy Ma and Karen E. Wood’s article “Admission Predictors for Success in a Mathematics
Graduate Program,” \citep{Ma} the authors seek to evaluate the statistical significance of certain factors that predict the success of graduating from a mathematics PhD program. Such factors that contribute to graduation include one’s major, GRE scores, and the ranking of one’s undergraduate institution. The authors encourage future researchers to continue the evaluation of the many aspects of graduate education. With such inspiration, this paper seeks to continue the discussion by evaluating not the predictors of graduating from a mathematics PhD program but rather the predictors of admissions into a program.

Ma and Wood find that at first one's undergraduate cumulative GPA is negatively correlated with success in a mathematics PhD program. This, of course, is an interesting result as a student's GPA should reflect the student's academic proficiency. However, once Ma and Wood account for the tier of the applicant's undergraduate school, they observed that tier is more indicative of success in a mathematics PhD program than GPA. Based on this finding, Ma and Wood infer that higher ranked schools hold a greater level of mathematical rigor, therefore resulting in a lower GPA.

Additionally, the two authors also observe there is a positive correlation between the verbal GRE score and success in a math PhD program. In fact, the two find that the verbal score provides a greater correlation with graduating from a mathematics PhD program than the quantitative scores. This suggests that the verbal GRE score is a better predictor of success in a math PhD program than the quantitative GRE score.

The fact that tier and the standardized testing methods are predictors of success in a mathematics PhD program will be evaluated in our study. The reader will later see how much emphasis graduate admission committees place on these criteria and how they factor into the admission's process.

The next stop on our guided tour is with Joseph A. Gallian's "A History of Undergraduate Research in Mathematics" \citep{Gallian}. As former president of the Mathematics Association of America, Gallian provides a brief chronicle of mathematical research at the undergraduate level. The key take away from this article is that although undergraduate research is widely accepted and prominent today, this was not the case 40 years ago. 

Discussion on undergraduate involvement in mathematics research began seriously in the late 1950s, but was met with much opposition and skepticism amongst mathematicians. This opposition was primarily due to the fact that research in mathematics necessitates an understanding of mathematics at a level way beyond an undergraduate's. Then Mathematical Association of America president Lynn Steen summarizes the leading thought in the February, 1986 \textit{FOCUS} newsletter stating that

\begin{quote}
    "Typically, good undergraduates glimpse the frontiers of
science from association with faculty research projects.
However, research in mathematics is not like research in the
laboratory sciences. Whereas undergraduates can become
apprentice scientists in chemistry research laboratories,
research in mathematics is so far removed from the undergraduate curriculum that little if any immediate benefit to
the undergraduate program ever trickles down from faculty
research. As a general rule, undergraduates can neither participate in nor even understand the research activity of their
mathematics professors." \citep{MAAFocus}
\end{quote}

Nevertheless, although intense skepticism still persisted, undergraduate research in mathematics became more prevalent starting in 1987 when the National Science Foundation (NSF) funded the first eight Research Experience for Undergraduates (REU) programs.

However, today's math REU programs are much different than those in their humble beginnings, and it wasn't until much later that programs started to resemble our current REU programs. The February 1987 \textit{Notices of the American Mathematical Society} states that REU participants are only responsible for "generating data, working out examples in order to develop conjectures, or performing literature reviews" \citep{AMSNotice}. Gallian rightfully notes that "most people who run REUs would consider these activities as the starting point, not the end product of an REU." Such differences exemplify the progression and development of REU programs. As many academics consider research experience to be an important factor in graduate admissions, the recency of undergraduate research in mathematics outlined by Gallian will serve as contextual background on why certain results occur in our findings.

Finally, some statistics regarding diversity within mathematics education will provide the reader with background on inclusive excellence. It is clear that there is much to do to improve inclusive excellence in the math community. In fact, the 2017-2018 \textit{Mathematical and Statistical Sciences Annual Survey} from the American Mathematical Society states that only $8\%$ of new doctorates were from minority groups. The survey also states that women accounted for $29\%$ of graduates recipients in 2017, the third year of a consecutive decrease in women doctoral recipients \citep{AMSAnnualReport}.

Data from the NSF also shows that although there has been a steady increase amongst Hispanics in mathematics bachelor degrees, rates amongst African-Americans are on a decline \citep{NSFMinorityData}. Much more data regarding the state of minority participation in mathematics can be presented, but the above statistics should suffice to realize that diversity is an issue within mathematics.

\begin{figure}
    \centering
    \includegraphics[scale=.3]{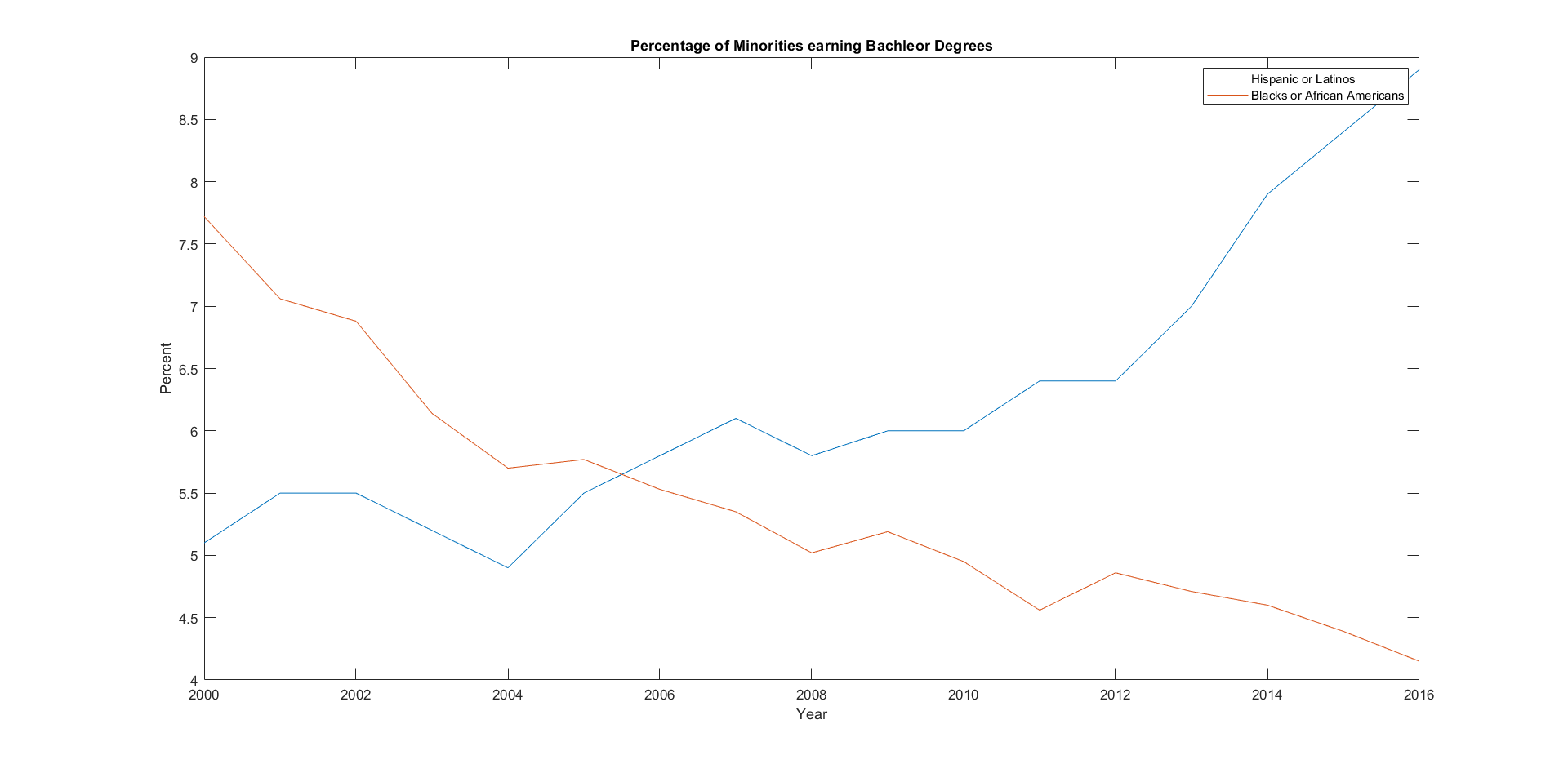}
    \caption{Percentages of Minorities Earning Bachelor Degrees}
    \label{figPercentages of Minorities earning Bachelor Degrees}
\end{figure}

\subsection{On The Analytic Hierarchy Process}

Thomas L. Saaty's Analytic Hierarchy Process (AHP) is a conjoint analysis tool widely used to make multi-criteria decisions based on a quantitative ranking. Saaty describes it as "a multi-valued logic. The AHP scale admits different intensities and captures priorities that indicate a range of possibilities for our preferences" \citep{TheoryandApplications}. Although we provide its theoretical foundations in the next section, the following contextual background will help with understanding Saaty's AHP.

The AHP structure is divided into three layers: the goal, the criteria, and the alternative choice. The criteria layer contains the factors the decision maker considers when attempting to achieve his goal, and the alternative choice layer contains the different options he has. Through a series of pairwise comparisons, the criteria are quantitatively assigned a relative importance value. Each choice is likewise assigned a relative importance value amongst each other with respect to each of the criteria. The best applicant is determined by weighing their values with respect to the overall goal. This sort of system, then, creates a ranking amongst applicants by taking into account all criteria. A visual representation with 5 criteria and 3 options is displayed in figure \ref{fig:Analytic Hierarchy Process Structure}.

\begin{figure}
    \centering



\tikzstyle{cloud} = [draw, ellipse,fill=blue!20, node distance=2.5cm,minimum height=2em]
\tikzstyle{io} = [trapezium, trapezium left angle=70, trapezium right angle=110,minimum height=2em,text centered, draw=black, fill=blue!30]
\tikzstyle{decision} = [diamond, draw, fill=red!20, text width=5em, text badly centered, node distance=3cm, inner sep=5pt]
\tikzstyle{block} = [rectangle, draw, fill=blue!20, text width=5em, text centered, rounded corners, minimum height=4em]
 \tikzstyle{line} = [draw, -latex']  

\begin{center}
    \begin{tikzpicture}[node distance = 2cm, auto]
    \node [decision] (Goal) {Goal};

    \node [block, below of=Goal, node distance = 4cm] (Criterion 3) {Criterion 3};
    \node [block, right of=Criterion 3, node distance = 2.5cm] (Criterion 4) {Criterion 4};
    \node [block, right of=Criterion 4, node distance = 2.5cm] (Criterion 5 ) {Criterion 5 };
    \node [block, left of= Criterion 3, node distance = 2.5cm] (Criterion 2) {Criterion 2};
    \node [block, left of=Criterion 2, node distance = 2.5cm] (Criterion 1) {Criterion 1};
    
    \node [cloud, below of=Criterion 3, node distance = 4cm] (Option1) {Option 2};
    \node [cloud, right of=Option1, node distance = 4cm] (Option2) {Option 3};
    \node [cloud, left of=Option1, node distance = 4cm] (Option3) {Option 1};

    \path [line] (Goal) -- (Criterion 3);
    \path [line] (Goal) -- (Criterion 4);
    \path [line] (Goal) -- (Criterion 5 );
    \path [line] (Goal) -- (Criterion 2);
    \path [line] (Goal) -- (Criterion 1);
    
    \path [line] (Criterion 3) -- (Option1);
    \path [line] (Criterion 4) -- (Option1);
    \path [line] (Criterion 5 ) -- (Option1);
    \path [line] (Criterion 2) -- (Option1);
    \path [line] (Criterion 1) -- (Option1);
    
    \path [line] (Criterion 3) -- (Option2);
    \path [line] (Criterion 4) -- (Option2);
    \path [line] (Criterion 5 ) -- (Option2);
    \path [line] (Criterion 2) -- (Option2);
    \path [line] (Criterion 1) -- (Option2);
    
    \path [line] (Criterion 3) -- (Option3);
    \path [line] (Criterion 4) -- (Option3);
    \path [line] (Criterion 5 ) -- (Option3);
    \path [line] (Criterion 2) -- (Option3);
    \path [line] (Criterion 1) -- (Option3);
\end{tikzpicture}
\end{center}
    \caption{Analytic Hierarchy Process Structure}
    \label{fig:Analytic Hierarchy Process Structure}
\end{figure}
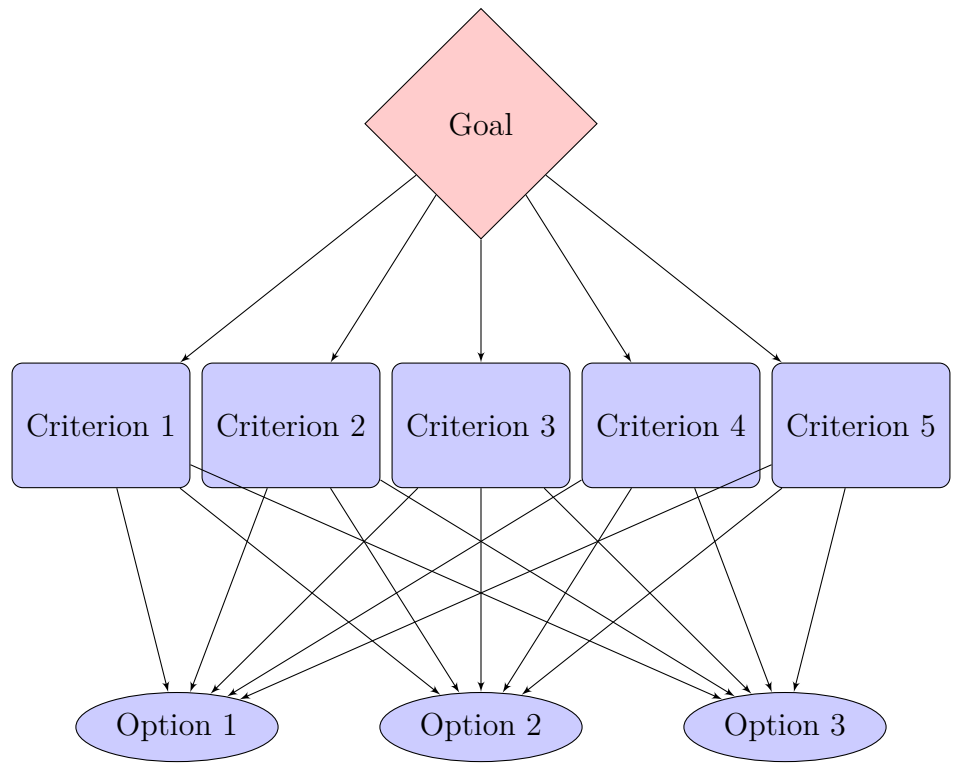

Rather than asking for a direct ordinal ranking, AHP uses a series of pairwise comparisons to evaluate the relationship between two variables. In this way, AHP is a conjoint analysis method used to construct a more accurate ordinal ranking based on the pairwise comparisons. Saaty writes that this higher degree of accuracy is due to the fact that "The most effective way to concentrate judgement is to take a pair of elements and compare them on a single property without concern for other properties or other elements" \citep{Saaty}. As such, it is not far-fetched to say that these pairwise comparisons can be considered the heart of Saaty's Analytic Hierarchy process.

The essence of this paper, then, is a multi-criteria decision analysis of graduate admissions. Although I will not be making any admission decisions, I will be using Saaty's to create a hierarchy (hence the name, Analytic \textit{Hierarchy} Process) of graduate admission criteria in mathematics.




\section{Theoretical Background}
Consider an individual comparing choices $C_1,C_2,..,C_n$ by pairwise comparison. The results of the individual's pairwise comparison are presented in the following matrix:

M
=

\begin{equation}\label{Matrix1}
\makeatletter\setlength\BA@colsep{10pt}\makeatother 
M=
\begin{blockarray}{c c c c c l}
        C_1 & C_2 & C_3 & C_4 & C_5 & \\
    \begin{block}{[@{\hspace{1pt}} c @{\hspace{1pt}} c @{\hspace{1pt}} c @{\hspace{1pt}} c @{\hspace{1pt}} c @{\hspace{1pt}}] @{\hspace{5pt}} l}
    \frac{w_1}{w_1} &  \frac{w_1}{w_2} & \hdots & \frac{w_1}{w_3} & \frac{w_1}{w_4} & C_1 \\
    \frac{w_2}{w_1} &  \frac{w_2}{w_2} & \hdots & \frac{w_2}{w_3} & \frac{w_2}{w_4} & C_2 \\
    \vdots & \vdots & \ddots & \vdots & \vdots & \vdots \\
    \frac{w_3}{w_1} &  \frac{w_3}{w_2} & \hdots & \frac{w_3}{w_3} & \frac{w_3}{w_4} & C_3 \\
    \frac{w_4}{w_1} &  \frac{w_4}{w_2} & \hdots & \frac{w_4}{w_3} & \frac{w_4}{w_4} & C_4 \\
    \end{block}
\end{blockarray}
\end{equation}
such that weight $w_i$ refers to choices $C_i \ \forall \ i=1,2,..,n$. Note that this produces a positive diagonal matrix with multiplicative inverses across the diagonal of ones. That is, $M_{ij}=\frac{1}{M_{ji}} \ \forall \ i,j=1,..,n$. The theory behind Thomas L. Saaty's Analytic Hierarchy Process is based on evaluating this pairwise comparison matrix. As a result, it is important to understand Saaty's definition of when a comparison matrix is consistent.

\begin{definition}[Consistent]
    A matrix M is consistent if $M_{ij}=M_{ik} M_{kj}$
\end{definition}

To demonstrate, suppose an individual has a preference intensity of $\zeta_1$ between choices $C_i$ and $C_j$, $\zeta_2$ between choices $C_j$ and choices $C_k$, and $\zeta_3$ between choices $C_i$ and $C_k$. We can write these pairwise comparisons as

\begin{equation}
        \frac{w_i}{w_j}= \zeta_1 \tab \label{Preference Intensity 1} \tag{Preference Intensity for $C_i$ to $C_j$}
\end{equation}
\begin{equation}
        \frac{w_j}{w_k}= \zeta_2 \tab \label{Preference Intensity 2} \tag{Preference Intensity for $C_j$ to $C_k$}
\end{equation}
\begin{equation}
    \frac{w_i}{w_k}= \zeta_3 \tab \label{Preference Intensity 3} \tag{Preference Intensity for $C_i$ to $C_k$}
\end{equation}
Rewriting $w_i=\zeta_1w_j$ and $w_k = \frac{w_j}{\zeta_2}$ and plugging these two values in the Preference Intensity for $C_i$ to $C_k$, we get 
$$\zeta_3 = \frac{\zeta_1w_j}{\frac{w_j}{\zeta_2}} = \zeta_1\zeta_2$$

This implies that when an individual is perfectly consistent, the preference intensity between any two choices $C_i$ and $C_k$ should be equal to the preference intensity between choices $C_i$ and $C_j$ times the intensity between choices $C_j$ and $C_k$. However, this is not necessarily true in practice. Take for instance the following example.
\begin{exmp*} Consider an individual trying to rank three types of candies: lollipops, taffy, and chocolate. Reconstructing matrix \ref{Matrix1}, we get
\begin{equation}\label{InconsistencyExample}
\kbordermatrix{
    &Lollipops&Taffy&Chocolate\\
    Lollipops&1&\frac{w_L}{w_T}&\frac{w_L}{w_C}\\
    Taffy&\frac{w_T}{w_L}&1&\frac{w_T}{w_C}\\
    Chocolate&\frac{w_C}{w_L}&\frac{w_C}{w_T}&1\\}
\end{equation}
Assume the individual provides a ranking in which he moderately prefers lollipops over taffy, strongly prefers chocolate over lollipops, and is indifferent between chocolate and taffy. Using the scale presented in table \ref{tab:StudyScaling}, the comparison matrix \ref{InconsistencyExample} is written as
\begin{equation}\label{InconsistencyExample2}
\kbordermatrix{
    &Lollipops&Taffy&Chocolate\\
    Lollipops&1&2&\frac{1}{3}\\
    Taffy&\frac{1}{2}&1&1\\
    Chocolate&3&1&1\\}
\end{equation}
Since $\frac{Lollipop}{Taffy} = \frac{2}{1}$ and $\frac{Taffy}{Chocolate} = \frac{1}{1}$, we expect the individual's preference intensity between lollipops and chocolate to be
$$\frac{Lollipop}{Chocolate} = \frac{Lollipop}{Taffy} \frac{Taffy}{Chocolate} = \frac{2}{1}\frac{1}{1}=2$$
However, the artificial example presents a preference of $\frac{1}{3}$. The individual's preference ranking can intuitively be seen to be inconsistent as if we wanted to present a strict ordinal ranking, we would run into contradictions as the individual prefers lollipops over taffy yet prefers chocolate, equivalent to taffy, over lollipops. This example portrays a matrix that is inconsistent by Saaty's definition.
\end{exmp*}
Such inconsistencies in one's preferences are the reasons why classic utility methods from economics are unsuitable for this study as such methods require the condition of transitivity. Saaty's Analytic Hierarchy Process, however, accounts for such intransitivities with the following measurements of a comparison matrix's consistency.

\begin{definition}[Consistency Index]
    The consistency index (CI) is the negative average of the other roots of the characteristic polynomial of the preference matrix \citep{Saaty}. This is denoted as
\begin{equation*}\label{ConsistencyIndex}
        CI = \frac{\lambda_{max}-n}{n-1}
    \end{equation*}
\end{definition}
\begin{definition}[Random Index]
    The random index (RI) is the average of the confidence indices computed from a large number of preference matrices of order $n \times n$ with random entries \citep{Saaty}. After running 50,000 random matrices, Thomas Saaty and Liem T Tran's obtain the random index in table \ref{tab:RandomIndexTable} of the appendix \citep{Saaty_Tran}.
\end{definition}
\begin{definition}[Consistency Ratio]
    The consistency ratio is the ratio between the consistency index and the random index. That is,
    $$CR = \frac{CI}{RI}$$
\end{definition}
Saaty determines the consistency index to be the difference between the largest eigenvalue of an inconsistent matrix and consistency index of a perfectly consistent matrix (equivalently the size of the matrix), divided by $n-1$. The consistency index is then compared to the random index defined above, resulting in the consistency ratio. Saaty determines that any consistency ratio less than $0.1$ is considered consistent \citep{Saaty}. A brief derivation is given in the appendix for further information.



Note that Saaty's Analytic Hierarchy Process additionally contains a third layer to rank option choices based on the criteria weights as described in the Literature Review. However, as the goal of this study is merely to rank the admission criteria, this paper will only be performing the first phase of Saaty's Analytic Hierarchy Process.

\section{Data and Methodology}
\setcounter{secnumdepth}{0}

\subsection{Variables}

For this study, 12 factors were slected to be ranked: Background, Major, Cumulative GPA, Math GPA, Research, Interview Performance, Upper Divsion Math Grades, Lower Division Math Grades, Quantitative GRE score, Verbal GRE Score, Math Subject GRE Score, Undergraduate Institution Tier. The following definitions of each variable were provided to the participants in our independent survey:

\begin{itemize}

    \item \textbf{Background/Demographics (Back)}: Survey participants were told to consider factors such as, but not limited to, race, gender, age, marital status and their department's policy on inclusive excellence.

    \item \textbf{Major}: Major is determined to be a variable in which one is a math major or non-math major.

    \item \textbf{Cumulative and Math GPA (CGPA/MGPA)}: Survey participants were told to assume that academic rigor and GPA is consistent between institutions. As a result, applicants assumed that GPA was not dependent upon an institution when comparing the relevance importance of admission criteria. This allows the criteria of CGPA/MGPA to be evaluated independent of grade inflation.
    

    \item \textbf{Research Experience}: Survey participants were told that research experience consists of substantial work in which the applicant produces insightful, though not necessarily novel and revolutionary, mathematical research.
    
    \item \textbf{Interview Performance}: Survey participants were told to consider factors such as, but not limited to, mathematical understanding, passion, research interests, and career goals.

    \item \textbf{Lower Division Math Grades (LDM)}: Survey participants were told to assume applicants have taken core lower division courses such as Single Variable Calculus, Multivariable Calculus, Linear Algebra, and Differential Equations.
    
    \item \textbf{Upper Division Math Grades (UDM)}: Survey participants were told to assume applicants have taken core upper division courses such as Advanced Calculus, Abstract Algebra, Linear Algebra, and Probability Theory.

    \item \textbf{Quantitative, Verbal, and Subject GRE Scores (GREQ/GREV/GRES)}: The subject math GRE score, quantitative GRE score, and verbal GRE score are determined by the Educational Testing Service.

    \item \textbf{School Tier}: Defined by the American Mathematical Society in their annual report \citep{AMSAnnualReport}. Survey participants were also told to consider the reputation of the mathematics' department.

\end{itemize}

\subsection{Analysis Methodology}

The goal was determined to be successful admission into a Mathematics PhD program based on the 12 criteria detailed above. In order to determine the relative weights of each criterion, the comparison matrix needs to be filled in as described in matrix \ref{Matrix1}. A reformatted table is shown in table \ref{tab:Criteria_Criter_Matrix} to provide a better visualization.

\begin{table}
    \centering
    \scalebox{0.6}
    {
    \begin{tabular}{l|c c c c c c c c c c c c |c|c}
         &Back&Major&CGPA&MGPA&Research&Interview&UDM&LDM&GREQ&GREV&GRES& School Tier &Sum&Weights\\
         \hline
         Back& 1& \textbf{(a)}& & & & & & & & & & & &\\
         Major& \textbf{(b)}& 1& & & & & & & & & & & &\\
         CGPA& & & 1& & & & & & & & & & &\\
         MGPA& & & & 1& & & & & & & & & &\\
         Research& & & & & 1& & & & & & & & &\\
         Interview& & & & & & 1& & & & & & & &\\
         UDM& & & & & & & 1& & & & & & &\\
         LDM& & & & & & & & 1& & & & & &\\
         GREQ& & & & & & & & & 1& & & & &\\
         GREV& & & & & & & & & & 1& & & &\\
         GRES& & & & & & & & & & & 1& & &\\
         School Tier& & & & & & & & & & & & 1& &\\
    \end{tabular}
    }
    \caption{Criteria Comparison Matrix}
    \label{tab:Criteria_Criter_Matrix}
\end{table}
In order to get the best estimate of the relative weights, surveys were sent out to mathematics faculty members across the United States. Their scores for each individual cell were then added up and divided by the number of responses to get the mean of each cell. The AHP analysis was ran on this set of numbers based on the response mean.

Each cell compares the relative importance between row criteria and column criteria. Using the scale in table \ref{tab:StudyScaling} of the appendix, survey participants were asked to compare all 12 criteria through a series of pairwise comparison questions via an online survey on Qualtrics. Each pairwise comparison question was presented individually for a total of 66 pairwise comparisons questions. The responses were then entered into an Excel spreadsheet as a data set. For instance, if a faculty member considers the criterion of Background moderately more important than one's Major, the marked $\textbf{(a)}$ in table \ref{tab:Criteria_Criter_Matrix} will be entered as a $3$. Since cell $\textbf{(b)}$ compares Cumulative GPA to Major, the cell is ipso facto ranked a $\frac{1}{3}$. 


To get the relative weights of each criteria, a row total for each criteria was calculated and then divided by the total across all criteria. The results are the two columns denoted "sum" and "weights" The weights in the right hand column signify the relative significance of each criteria. That is, the resulting numbers are the weights admission committees place on each criterion relative to the other criteria.

In addition to running the AHP model on the overall mean of all responses, the AHP model was ran on each response individually to produce a data set with the relative weights of each criterion for each response. Responses were then categorized by whether respondents have served on graduate admission committees or not (Yes, No, Blank), professorship rank (Full, Associate, Assistant, Visiting Assistant, Lecturer, Blank), and PhD program ranking (Group 1, Group 2, Group 3, Blank). Mathematics PhD programs that ranked in the top 30 were defined as Group 1, PhD programs ranked 31 to 50 as Group 2, and all others as Group 3. Rankings were determined by U.S. News \& World Report \citep{USWorldNewsRanking}. The schools and their categorization used in my study are provided in Table \ref{tab:School Categorization} in appendix.

Statistical tests were performed in MatLab and occasionally Stata. Dividing the responses by service on graduate admission committees, professorship rank, and program tier allowed analysis on the statistical differences between varying populations. ANOVA and two sample t-tests were conducted to find population differences. Pearson Correlation tests were conducted to find correlation between factor rankings.

Additionally, online interviews with professors were conducted to provide qualitative evidence and insight into our findings. In total, 13 professors from various institutions across the United States were interviewed. Interviewees were primarily presented preliminary findings and asked their personal opinions on the data. All quotations were used with their approval.


\section{Results}
\subsection{Ordinal Ranking}

A breakdown of the survey responses can be seen in table \ref{tab:Survey Response Breakdown}. An additional categorization with the number of responses from each school can be found in table \ref{tab:School Categorization} in the Appendix. Note that the lecturer who has served on a graduate admission's committee did so previously as an assistant professor at a different institution.

\begin{table}
    \centering
    \begin{tabular}{c l|c|c|c|c}
        \multicolumn{6}{c}{Served on Graduate Admission Committee}\\
        \multirow{7}*{\rotatebox{90}{Professorship Rank}}  
         && Yes& No & Blank &Sum\\
         \cline{2-6}
         &Full& 45&3&1&49 \\
         &Associate& 8&4&0&12\\
         &Assistant&10&7&0&17\\
         &Visiting Assistant&0&3&0&3\\
         &Adjuncts/Instructors&1&6&0&7\\
         &Blank&3&3&14&20\\
        \cline{2-6}
         &Sum&67&26&15&108
    \end{tabular}
    \caption{Survey Response Breakdown}
    \label{tab:Survey Response Breakdown}
\end{table}

The criteria comparison matrix for those who have served on graduate admission committees is provided in table \ref{tab:Criteria to Criteria Final}. The reader is reminded that the value provided compares the row factor to the column factor referring to the associate value in table \ref{tab:StudyScaling}. As a result, each cell will signify the relative importance the row variable has over the column variable. For instance, whether an applicant is a math major or not is slightly to moderately more (1.596) important than the applicant's background.

\begin{table}
    \centering
    \scalebox{0.7}
    {
    \begin{tabular}{l|c c c c c c c c c c c c}
        &Back&Major&CGPA&MGPA&Research&Interview&UDM&LDM&GREQ&GREV&GRES&Tier\\
        \hline
        Back&1&0.627&0.614&0.471&0.648&0.878&0.452&0.850&0.867&1.111&0.782&0.827\\
        Major&1.596&1&1.353&0.835&0.949&1.198&0.648&1.020&1.261&1.523&0.972&1.159\\
        CGPA&1.628&0.739&1&0.568&0.769&1.148&0.514&0.815&1.117&1.432&0.844&1.082\\
        MGPA&2.125&1.198&1.761&1&1.554&1.806&0.931&1.662&1.789&2.182&1.465&1.559\\
        Research&1.543&1.054&1.301&0.643&1&1.659&0.951&1.65&1.755&2.015&1.377&1.511\\
        Interview&1.139&0.835&0.871&0.554&0.603&1&0.651&0.992&1.258&1.579&0.997&1.122\\
        UDM&2.212&1.543&1.946&1.075&1.052&1.536&1&2.506&2.191&2.372&1.731&1.813\\
        LDM&1.176&0.980&1.227&0.602&0.606&1.008&0.399&1&1.361&1.670&0.941&1.150\\
        GREQ&1.153&0.793&0.895&0.559&0.570&0.795&0.456&0.735&1&1.659&0.765&1.077\\
        GREV&0.900&0.657&0.698&0.458&0.496&0.633&0.422&0.599&0.603&1&0.597&0.747\\
        GRES&1.278&1.029&1.185&0.683&0.726&1.003&0.578&1.062&1.306&1.674&1&1.579\\
        Tier&1.209&0.863&0.924&0.642&0.662&0.891&0.551&0.870&0.928&1.339&0.633&1\\
    \end{tabular}
    }
    \caption{Criteria Comparison Matrix}
    \label{tab:Criteria to Criteria Final}
\end{table}

The final rankings for all samples and for only those who have served on a graduate admission committee are respectively provided in table \ref{tab:All Samples and All Factors} and table \ref{tab:Served on Graduate Committee and All Factors}. A summary of the statistics for those who have served on a graduate admission committee is provided in table \ref{tab:Statistics of Variables}. Both Analytic Hierarchy Processes also have a consistency ratio below $0.1$ indicating that the pairwise comparisons from the study'/s survey are consistent. Future tables will contain data from only those who have served on graduate admissions committees unless otherwise stated.

\begin{table}[]
    \begin{minipage}{.5\linewidth}
      \centering
        \def\arraystretch{.8}
        \begin{tabular}{c|c|m{2cm}}
        Rank & Factor & Relative Importance\\
        \hline
        1&UDM&13.415 \%\\
        2&MGPA&12.17 \%\\
        3&Research&10.525 \%\\
        4&Major&8.643 \%\\
        5&GRES&8.38 \%\\
        6&LDM&7.751 \%\\
        7&CGPA&7.454 \%\\
        8&Interview&7.419 \%\\
        9&Tier&6.723 \%\\
        10&GREQ&6.688 \%\\
        11&Back&5.838 \%\\
        12&GREV&4.995 \%\\
        \hline
        &Consistency Ratio&0.0046145
    \end{tabular}
    \caption{All Responses}
    \label{tab:All Samples and All Factors}
    \end{minipage}%
    \begin{minipage}{.5\linewidth}
      \centering
      \def\arraystretch{.8}
      \begin{tabular}{c|c|m{2cm}}
        Rank & Factor & Relative Importance\\
        \hline
        1&UDM&13.202 \%\\
        2&MGPA&12.108 \%\\
        3&Research&10.447 \%\\
        4&Major&8.647 \%\\
        5&GRES&8.248 \%\\
        6&CGPA&7.767 \%\\
        7&LDM&7.466 \%\\
        8&Interview&7.268 \%\\
        9&GREQ&6.876 \%\\
        10&Tier&6.572 \%\\
        11&Back&6.131 \%\\
        12&GREV&5.268 \%\\
        \hline
        &Consistency Ratio&0.0045437
    \end{tabular}
    \caption{Served on Graduate Committee}
    \label{tab:Served on Graduate Committee and All Factors}
    \end{minipage} 
    \label{tab:Final Rankings}
\end{table}

\begin{table}[!htp]
    \centering
    \scalebox{0.7}
    {
    \begin{tabular}{l|c c c c c c c c c c c c c}
        &Back&Major&CGPA&MGPA&Research&Interview&UDM&LDM&GREQ&GREV&GRES&Tier\\
        \hline
        Mean&5.2 \%&7.6 \%&7.0 \%&10.8 \%&9.7 \%&7.3 \%&12.5 \%&7.7 \%&7.8 \%&6.3 \%&9.9 \%&8.1 \%\\
        Standard Deviation&2.7 \%&3.0 \%&2.0 \%&2.3 \%&3.1 \%&3.2 \%&2.2 \%&2.3 \%&2.7 \%&2.7 \%&3.0 \%&2.6 \%\\
        Maximum&11.5 \%&15.6 \%&12.3 \%&16.3 \%&16.8 \%&17.0 \%&18.2 \%&12.6 \%&12.8 \%&16.9 \%&14.8 \%&14.6 \%\\
        Minimum&2.2 \%&3.2 \%&3.3 \%&4.3 \%&4.7 \%&2.2 \%&8.4 \%&3.1 \%&3.0 \%&2.6 \%&3.1 \%&2.7 \%\\
    \end{tabular}
    }
    \caption{Summary of Statistics for those who have served on a graduate admission committee}
    \label{tab:Statistics of Variables}
\end{table}

\subsection{Population Differences}

The following section outlines differences between different groups. It is organized such that differences between the types of professors are first described, followed by the differences between those who have served on a graduate admission committee or not, and finally differences between PhD program groups.

\subsubsection{Professorship Rank}

There is firstly a statistically significant difference (\textit{p =.028, t(86) = -2.235}) between how full professors and other faculty (associate professors, assistant professors, visiting assistant professors and lecturers) value lower division coursework. Full professors ranked lower division classes lower (\textit{M = .075, SD = .020}) compared to the other instructors (\textit{M = .085, SD = .024}). This means that full professors on average value lower division coursework at 7.5\% whereas other instructors value them at 8.5\%. When comparing just lecturers to full professors, the gap is even larger (\textit{p =.002, t(54) = -3.30}) with lecturers valuing lower division courses at 10.13\% \textit{(SD=.0198)} and full professors at 7.47\% \textit{(SD=.020)}. Note that this analysis contains all samples, including those who have not served on a graduate admission committee. 

\begin{figure}[!htp]
    \centering
    \includegraphics[scale=.3]{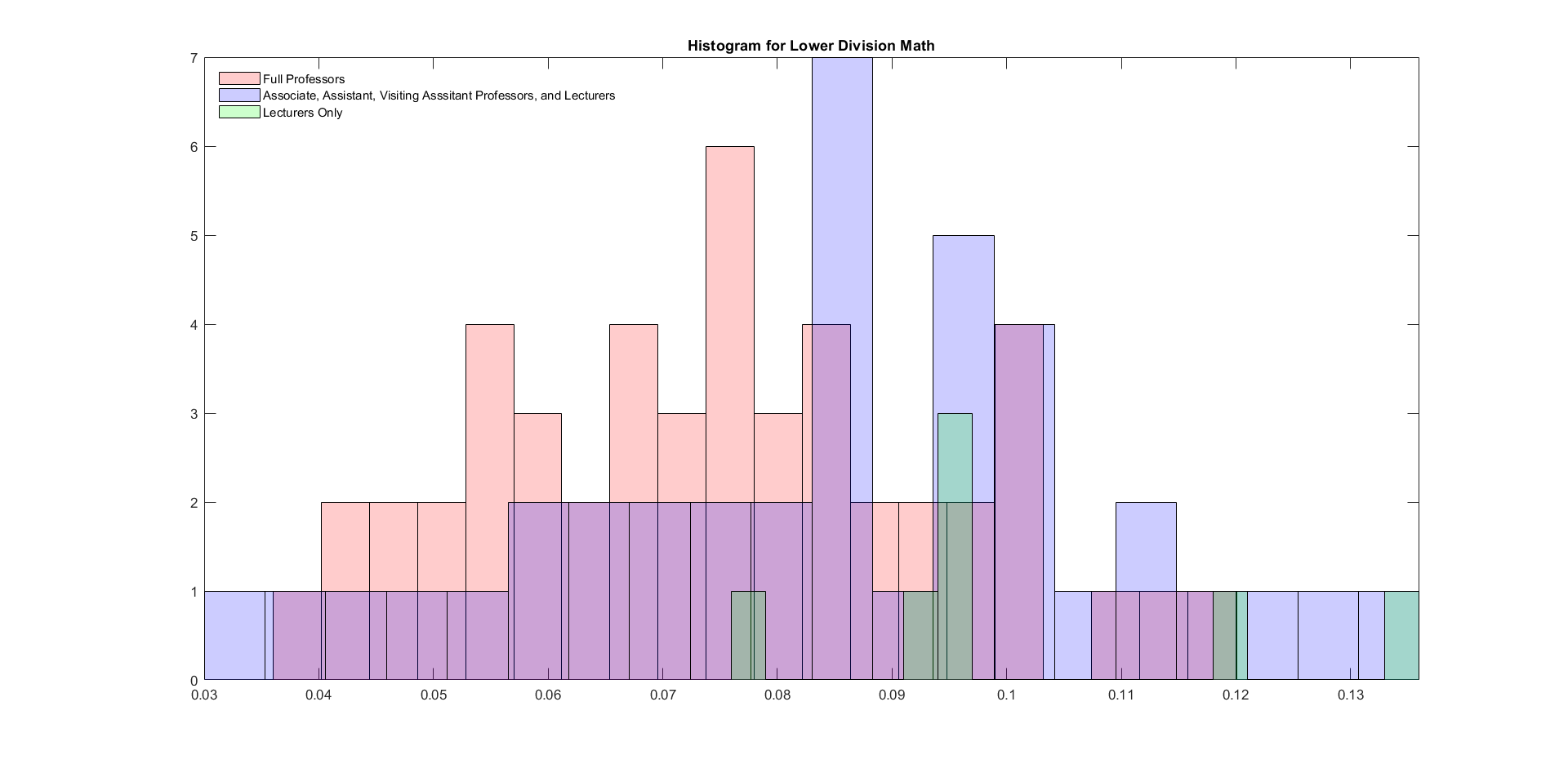}
    \caption{Histogram: Lower Division Math Coursework by Professorship Rank}
    \label{fig:Lower Division Population Difference 1}
\end{figure}

There is also a statistically significant difference (\textit{p =.005, t(76) = 2.879}) between how full professors and how associate/assistant professors collectively value the quantitative GRE. A two sample t-test shows that full professors statistically rank the quantitative GRE higher at 8.2\% \textit{(SD = .029)} than their associate and assistant colleagues at 6.5\%  \textit{(SD = .021)}. This difference also occurs with the math subject GRE score, but is only marginally statistically significant at the $.05\%$ level (\textit{p =.056, t(76) = 1.94}). Whereas full professors value the math subject GRE exam at 10.21\% (\textit{SD=.030}), associate and assistant professors value the exam at 8.89\% (\textit{SD=.026}).

\begin{figure}[!htp]
    \centering
    \includegraphics[scale=.3]{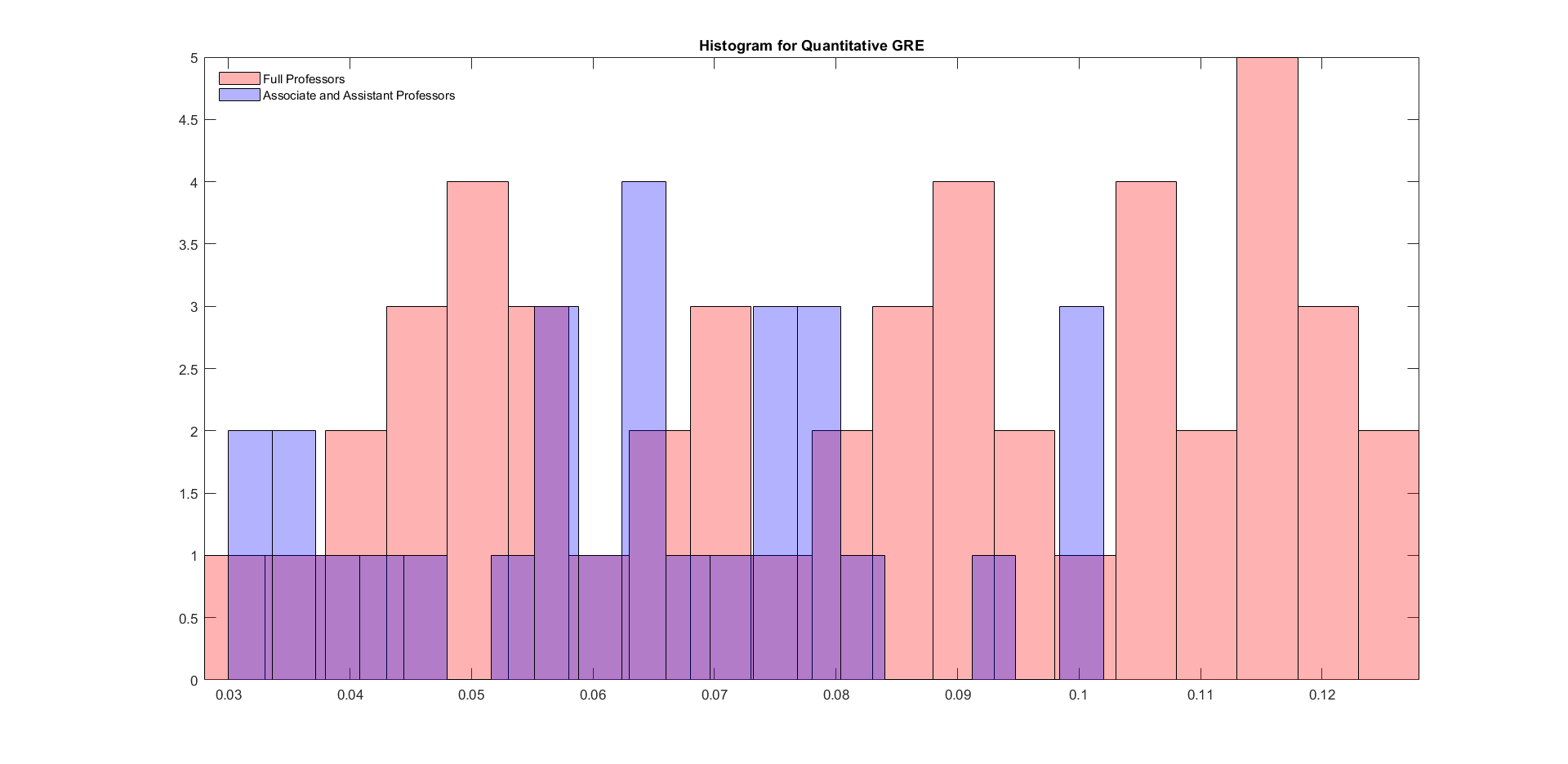}
    \caption{Histogram: Quantitative GRE by Professorship Rank}
    \label{fig:Quantitative GRE Population Difference}
\end{figure}

\begin{figure}[!htp]
    \centering
    \includegraphics[scale=.3]{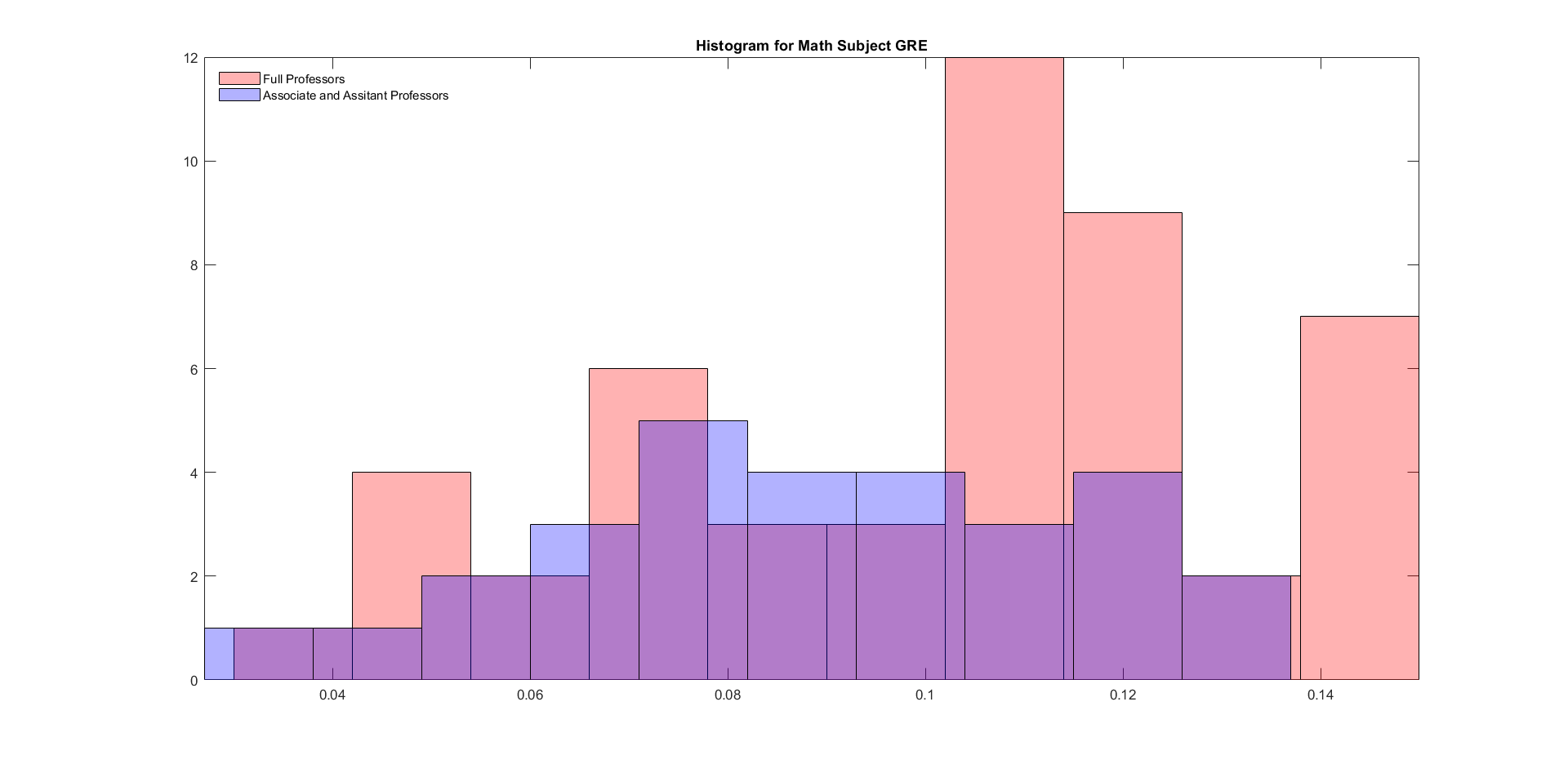}
    \caption{Histogram: Math Subject GRE by Professorship Rank}
    \label{fig:Math Subject GRE Population Difference}
\end{figure}

When testing between full and assistant professors, there is a statistical difference on how the two populations rank research experience (\textit{p =.050, t(64) = -2.00}). On average, full professors (\textit{M = .956, SD = .031}) tend to value research experience lower than assistant professors (\textit{M = .114, SD = .035}). However, this difference does not occur when testing between full and associate professors (\textit{p =.910, t(59) = 0.114}).

\begin{figure}[!htp]
    \centering
    \includegraphics[scale=.3]{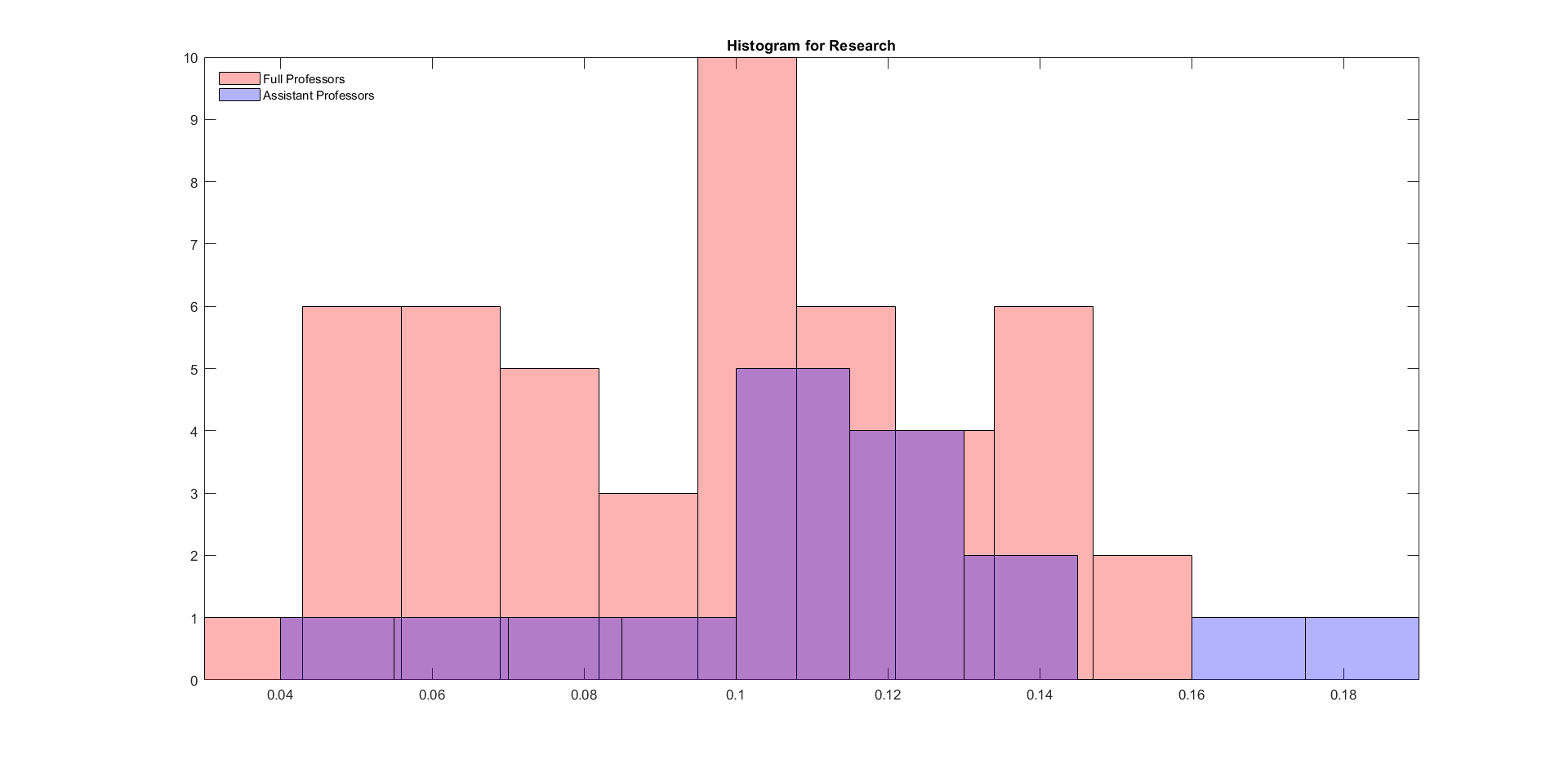}
    \caption{Histogram: Undergraduate Research Experience by Professorship Rank}
    \label{fig:Full vs. Assistant Professors over Research Histogram}
\end{figure}

At the inception of this study, there was a hypothesis that older full professors would rank an applicant's background lower compared to younger associate and assistant professors. However, there were no statistically significant differences \textit{(F(2,75)=.53, p =.591)} between the three types of professors. Full professors ranked background at 5.25\% \textit{(SD=.030}), associate professors ranked background at 4.75\% \textit{(SD=.026}) and assistant professors at 5.83\% \textit{(SD=.026}).

\begin{figure}[!htp]
    \centering
    \includegraphics[scale=.3]{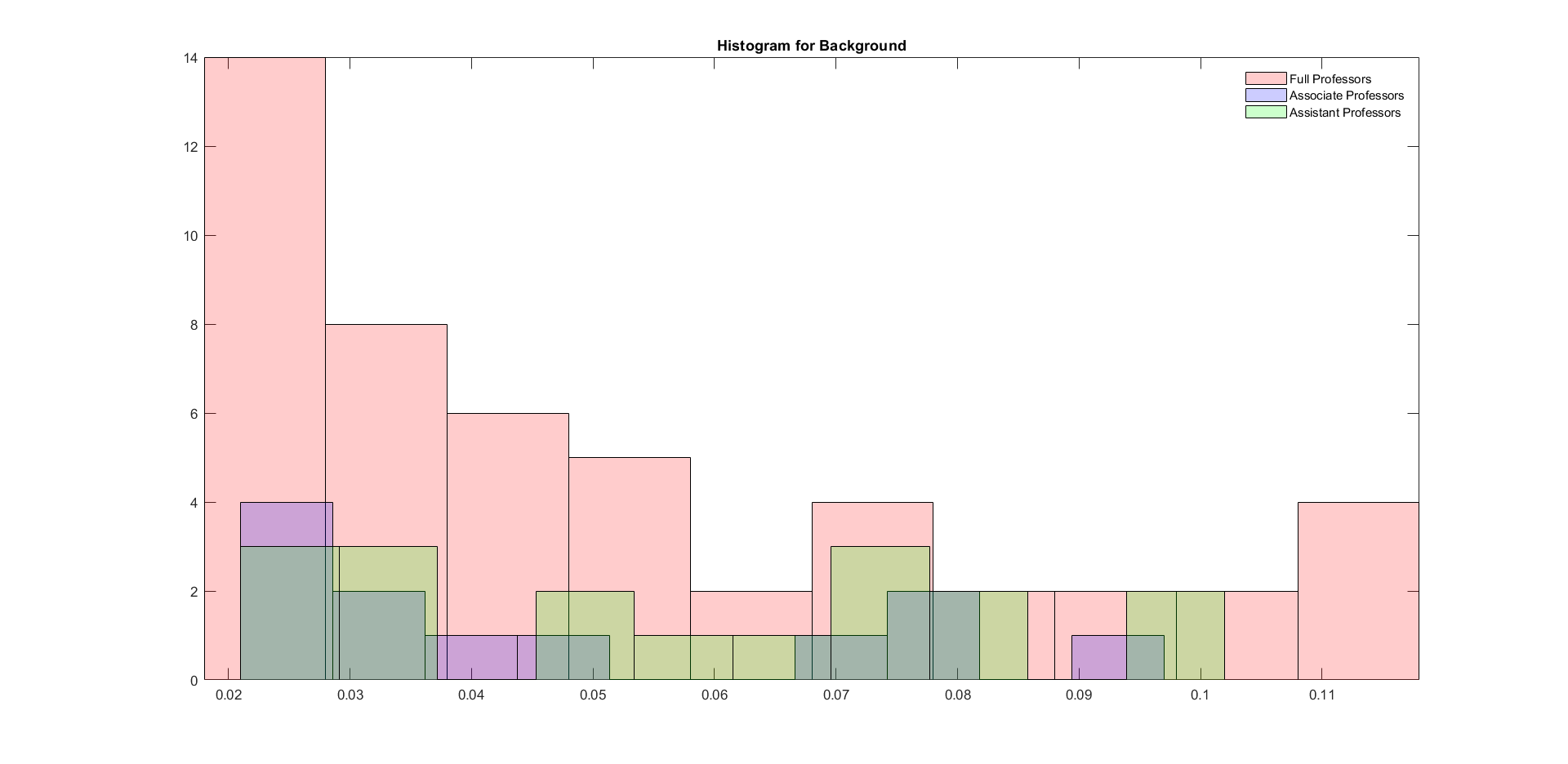}
    \caption{Histogram: Background by Professorship Rank}
    \label{fig:Full vs. Associate and Assistant for Background}
\end{figure}

\subsubsection{Committee Service}

There were no statistically significant differences when partitioning by those who have served on graduate admission committees or not. Table \ref{tab:Differences in those who have/have not served on Graduate Admission Committees} contains the relevant statistics, of which lower division coursework is the closest factor to a statistically significant difference.

\begin{table}[!htp]
    \centering
    \scalebox{0.7}
    {
    \begin{tabular}{l|c c c c c c c c c c c c c}
        &Back&Major&CGPA&MGPA&Research&Interview&UDM&LDM&GREQ&GREV&GRES&Tier\\
        \hline
        Mean of Yes Committee&0.052&0.076&0.070&0.108&0.097&0.073&0.125&0.077&0.078&0.063&0.099&0.081\\
        Mean of No Committee&0.048&0.075&0.067&0.113&0.1&0.075&0.128&0.086&0.07&0.053&0.103&0.08\\
        SD of Yes Committee&0.027&0.03&0.02&0.023&0.031&0.032&0.022&0.023&0.027&0.027&0.03&0.026\\
        SD of No Committee&0.026&0.025&0.017&0.017&0.032&0.026&0.02&0.019&0.02&0.017&0.027&0.025\\
        t-stat&0.549&0.155&0.709&-0.987&-0.473&-0.285&-0.725&-1.725&1.404&1.759&-0.602&0.157\\
        p-value&0.584&0.877&0.480&0.326&0.637&0.776&0.470&0.088&0.164&0.082&0.549&0.876\\
    \end{tabular}
    }
    \caption{Differences in those who have/have not served on Graduate Admission Committees}
    \label{tab:Differences in those who have/have not served on Graduate Admission Committees}
\end{table}

\subsubsection{Program Rankings}

When dividing the PhD programs into three Groups as mentioned in the Methodology section, Group 1 (the top 30 schools) valued background statistically less than schools outside of Group 1 (\textit{p =.033, t(74) = -2.18}). Group 1 schools ranked background with a mean of 3.7935\% (\textit{SD = .015}) while schools outside ranked background with a mean of 5.4252\% (\textit{SD = .029}). When testing individually between Groups, there is a smaller gap between Group 1 and Group 2 (\textit{p =.058, t(38) = -1.952}), with Group 1's mean at 3.7935\% \textit{(SD= .0152)} and Group 2 at 5.2868\% \textit{(SD= .028)}. However, the gap becomes more apparent when testing between Group 1 and Group 3 (\textit{p =.034, t(50) = -2.183}) with Group 3's average at 5.5174\% \textit{(SD= .030)}. There is no statistically significant difference between Group 2 and Group 3's value on background (\textit{p =.764, t(58) = -.301}).

\begin{figure}[!htp]
    \centering
    \includegraphics[scale=.3]{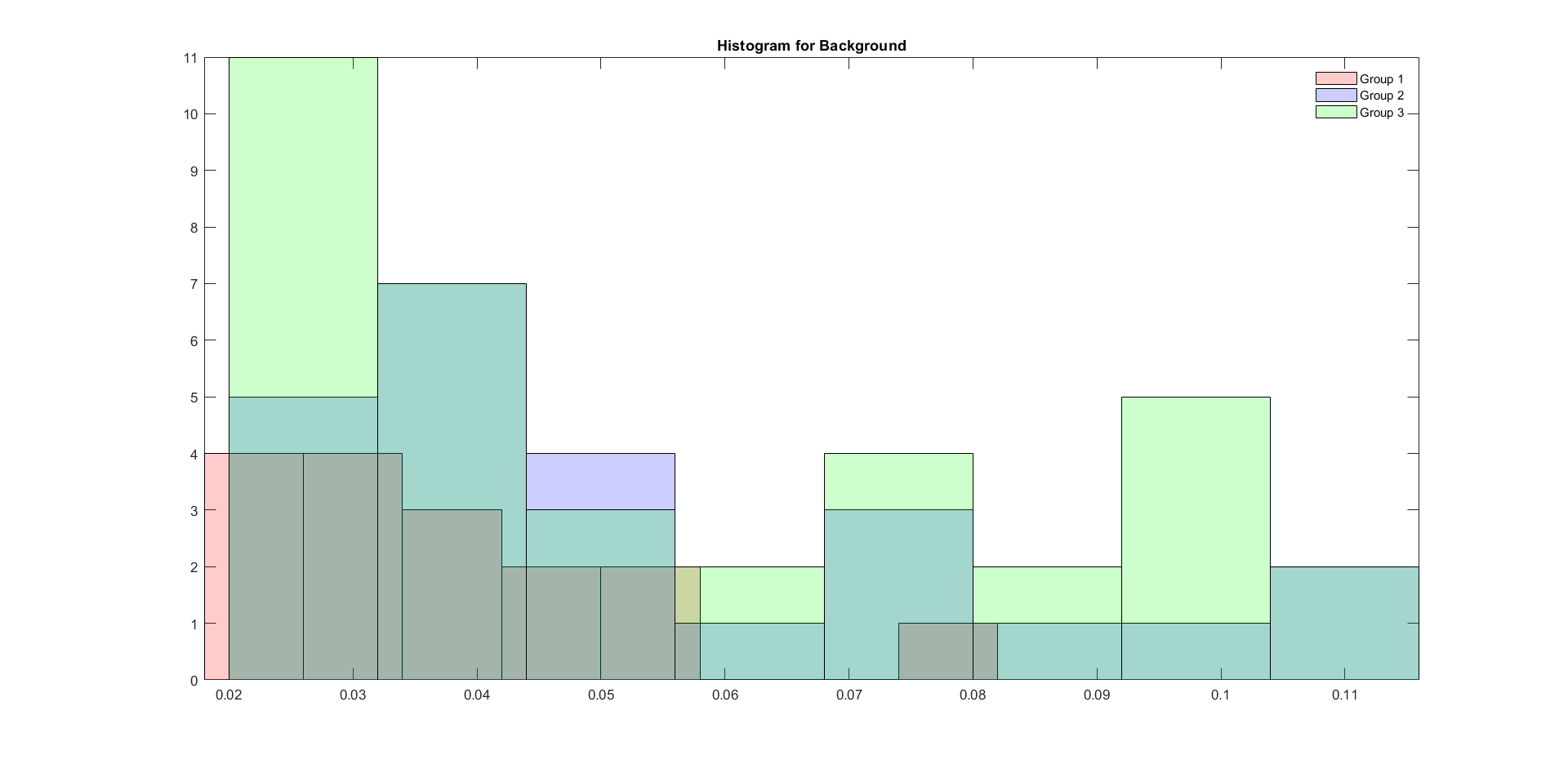}
    \caption{Histogram: Background by Groups}
    \label{fig:Background by Groups} \end{figure}

There were also statistically significant differences in how the groups valued the math subject GRE. Results show that the top 30 schools value the math subject GRE much higher than schools ranked 31-50 (\textit{p =.017, t(38) = 2.492}) and schools ranked over 51 (\textit{p =.024, t(50) = 2.324}). Group 1 valued the math subject GRE at 11.781\% \textit{(SD= .026)}, Group 2 at 9.6388 \% \textit{(SD= .096)}, and Group 3 at 9.8128 \% \textit{(SD= .029)}. There are no statistically significant differences between Group 2 and Group 3 in regards to the math subject GRE (\textit{p =.817, t(58) = -.233}).

\begin{figure}[!htp]
    \centering
    \includegraphics[scale=.3]{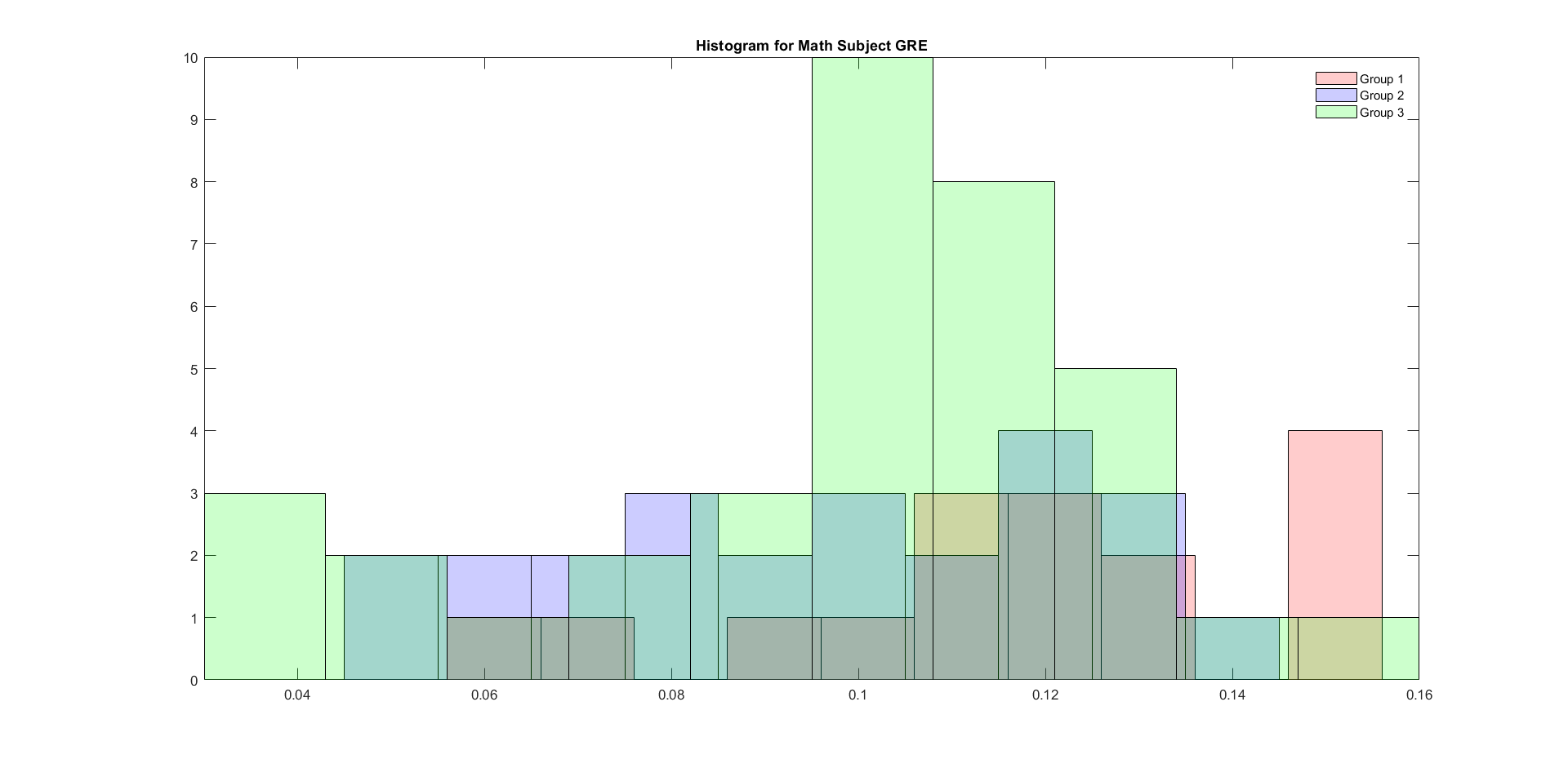}
    \caption{Histogram: Math Subject GRE by Groups}
    \label{fig:Math Subject GRE by Groups}
\end{figure}

Finally, the Groups also differed on how to rank an applicant's undergraduate school tier. Groups 1, 2, and 3 had a mean of 8.7418\% \textit{(SD =.029)}, 8.870 \% \textit{(SD =.026)}, and 7.182 \% \textit{(SD =.025)}, respectively. When testing for statistical significance, there is a difference between Groups 1 and 2 collectively vs. Group 3 (\textit{p =.007, t(74) = -2.768}). When testing individually, there were differences between Group 1 and Group 3 (\textit{p =.051, t(50) = 2.542}), Group 2 and Group 3 (\textit{p =.014, t(58) = 2.542}), but no differences between Group 1 and Group 2 (\textit{p =.884, t(38) = -.147}).

\begin{figure}[!htp]
    \centering
    \includegraphics[scale=.3]{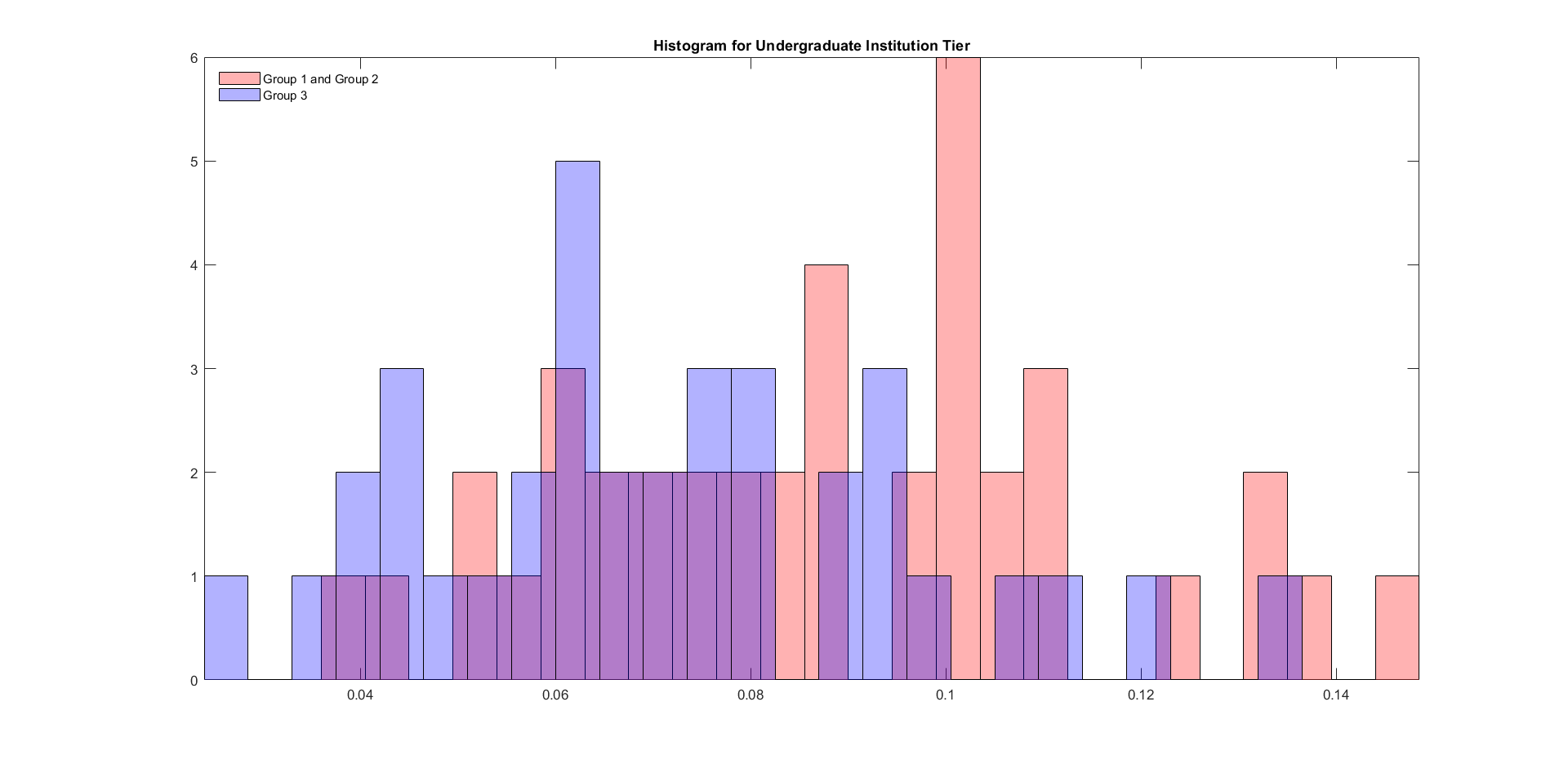}
    \caption{Histogram: Undergraduate School Tier GRE by Groups}
    \label{fig:Undergraduate School Tier GRE by Groups}
\end{figure}

\subsection{Correlations}

The following section outlines correlations between factors. All correlation coefficients can be found in Table \ref{tab:Correlation Coefficients} and the associated p-values in Table \ref{tab:p-values}. The correlations in this section will give insight into the relationship admission criteria have with each other. That is, the correlations in this section will answer the question "if a professor values admission factor $x$, how likely is he to value another admission factor $y$?" 
The following data again only contains responses from those who have served on mathematics PhD graduate admission committees.

\subsubsection{Background and Standardized Testing}

Background was found to be negatively associated with all three testing mediums: quantitative GRE scores (\textit{r(65) = -0.352, p =.003}), verbal GRE scores (\textit{r(65) = -0.254, p =.004}), and especially the math subject GRE scores (\textit{r(65) = -0.522, p =5.8E-06}). That is, professors who rank background higher tend to rank standardized testing methods lower, and professors who rank background lower tend to rank standardized testing methods higher. However, when testing between the three GRE scores internally, there exists moderate correlations between the three GRE scores. Quantiative GRE scores have a moderate correlation of (\textit{r(65) = 0.441, p =.0002}) with the verbal scores of the GRE and a moderate correlation (\textit{r(65) = 0.445, p =.0002}) with the subject GRE. However, the same can not be said between the verbal and the quantitative GRE (\textit{r(65) = 0.198, p =.379}).

\begin{figure}[!htp]
    \centering
    \includegraphics[scale=.28]{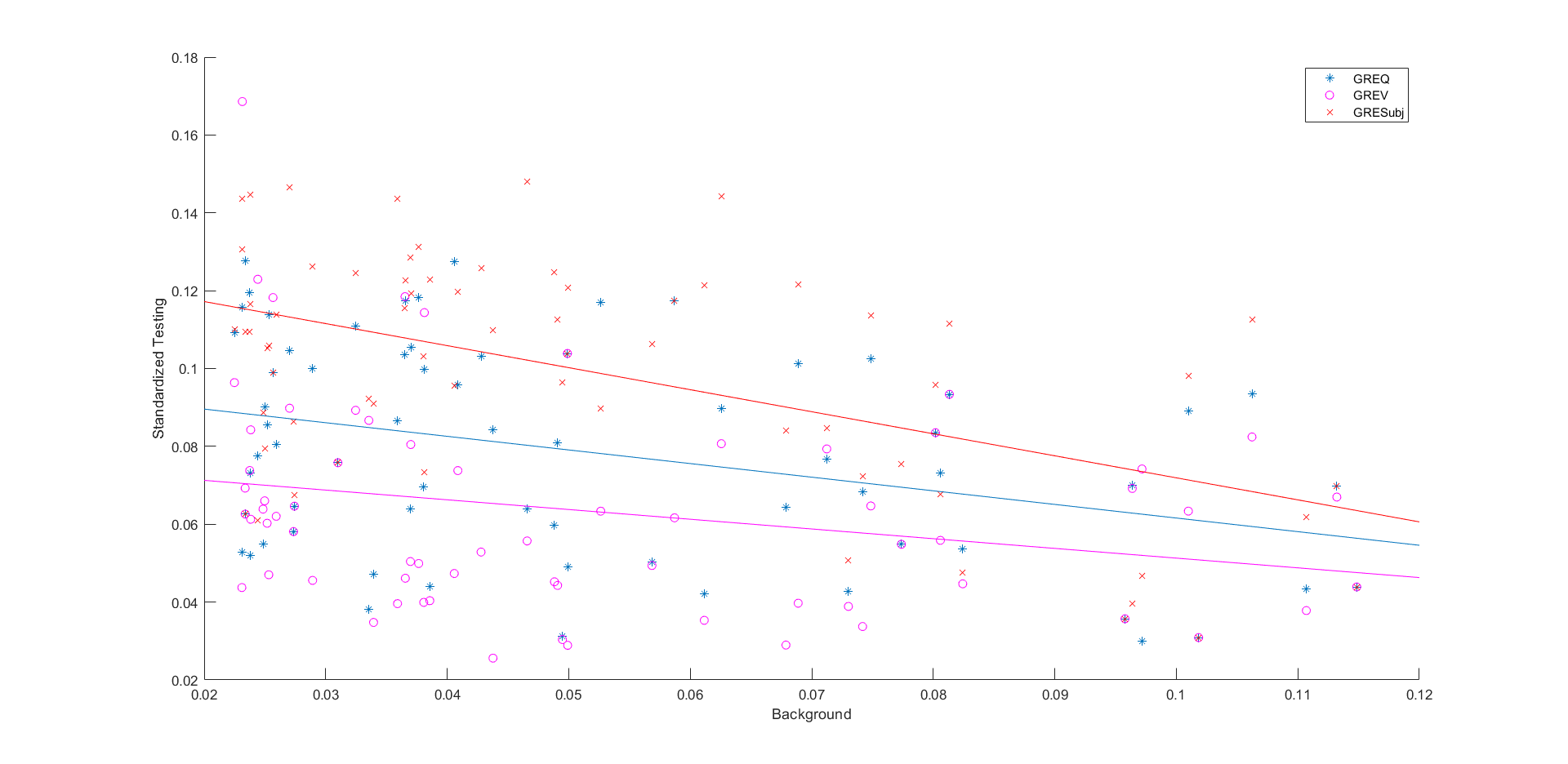}
    \caption{Correlation: Background and Standardized Testing}
    \label{fig:Correlation: Background and Standardized Testing Scatter}
\end{figure}

\afterpage{
\begin{table}[p]
    \centering
    \scalebox{0.7}
    {
    \begin{tabular}{l|c c c c c c c c c c c c c}
        &Back&Major&CGPA&MGPA&Research&Interview&UDM&LDM&GREQ&GREV&GR S&Tier\\
        \hline
        Back&1&0.135&0.011&0.005&0.042&0.207&0.125&-0.262&-0.352&-0.254&-0.522&-0.167\\
        Major&0.135&1&0.161&-0.034&-0.147&-0.02&-0.015&-0.391&-0.337&-0.224&-0.122&-0.102\\
        CGPA&0.011&0.161&1&0.234&-0.068&-0.113&-0.106&0.005&-0.281&-0.19&-0.301&-0.034\\
        MGPA&0.005&-0.034&0.234&1&-0.181&-0.454&0.451&0.304&-0.214&-0.39&-0.061&-0.225\\
        Research&0.042&-0.147&-0.068&-0.181&1&0.206&-0.077&-0.124&-0.372&-0.312&-0.179&-0.009\\
        Interview&0.207&-0.02&-0.113&-0.454&0.206&1&-0.407&-0.221&-0.195&0.006&-0.283&-0.103\\
        UDM&0.125&-0.015&-0.106&0.451&-0.077&-0.407&1&0.328&-0.145&-0.412&-0.142&-0.246\\
        LDM&-0.262&-0.391&0.005&0.304&-0.124&-0.221&0.328&1&0.015&-0.048&-0.028&-0.232\\
        GREQ&-0.352&-0.337&-0.281&-0.214&-0.372&-0.195&-0.145&0.015&1&0.441&0.445&-0.051\\
        GREV&-0.254&-0.224&-0.19&-0.39&-0.312&0.006&-0.412&-0.048&0.441&1&0.109&0.151\\
        GRES&-0.522&-0.122&-0.301&-0.061&-0.179&-0.283&-0.142&-0.028&0.445&0.109&1&-0.039\\
        Tier&-0.167&-0.102&-0.034&-0.225&-0.009&-0.103&-0.246&-0.232&-0.051&0.151&-0.039&1\\
    \end{tabular}
    }
    \caption{Pearson's Correlation Coefficients}
    \label{tab:Correlation Coefficients}
\end{table}
\begin{table}[p]
    \centering
    \scalebox{0.7}
    {
    \begin{tabular}{l|c c c c c c c c c c c c c}
        &Back&Major&CGPA&MGPA&Research&Interview&UDM&LDM&GREQ&GREV&GRES&Tier\\
        \hline
        Back&1&0.277&0.93&0.968&0.737&0.092&0.313&0.032&0.003&0.038&0&0.177\\
Major&0.277&1&0.192&0.783&0.236&0.872&0.906&0.001&0.005&0.068&0.326&0.41\\
CGPA&0.93&0.192&1&0.057&0.584&0.364&0.392&0.969&0.021&0.124&0.013&0.785\\
MGPA&0.968&0.783&0.057&1&0.143&0&0&0.012&0.082&0.001&0.627&0.067\\
Research&0.737&0.236&0.584&0.143&1&0.095&0.537&0.317&0.002&0.01&0.147&0.942\\
Interview&0.092&0.872&0.364&0&0.095&1&0.001&0.072&0.114&0.96&0.02&0.408\\
UDM&0.313&0.906&0.392&0&0.537&0.001&1&0.007&0.242&0.001&0.252&0.045\\
LDM&0.032&0.001&0.969&0.012&0.317&0.072&0.007&1&0.901&0.703&0.82&0.059\\
GREQ&0.003&0.005&0.021&0.082&0.002&0.114&0.242&0.901&1&0&0&0.68\\
GREV&0.038&0.068&0.124&0.001&0.01&0.96&0.001&0.703&0&1&0.379&0.224\\
GRES&0&0.326&0.013&0.627&0.147&0.02&0.252&0.82&0&0.379&1&0.755\\
Tier&0.177&0.41&0.785&0.067&0.942&0.408&0.045&0.059&0.68&0.224&0.755&1\\
    \end{tabular}
    }
    \caption{P-values from Pearson's Correlation Test}
    \label{tab:p-values}
\end{table}
\clearpage
}


\subsubsection{GPA}

Results indicate that the weighting of an applicant's cumulative GPA is negatively correlated with the weighting of the quantitative GRE scores (\textit{r(65) = -0.281, p =.02}) and math subject GRE scores (\textit{r(65) = -0.301, p =.01}), but shares no relationship with the verbal GRE scores (\textit{r(65) = -0.190, p =.124}). Whereas cumulative GPA is negatively correlated with the quantitative GRE and math subject GRE but shares no relationship with the verbal GRE, the opposite is true for the math GPA; there exists no correlation between math GPA and the quantitative GRE (\textit{r(65) = -0.214, p =.082}) or math subject GRE (\textit{r(65) = -0.06, p =.627}), but there is a moderate negative correlation with the verbal GRE(\textit{r(65) = -0.390, p =.001}). Math GPA also shares a moderate positive correlation between upper division and lower division math course grades (\textit{r(65) = -0.451, p =.0001} and \textit{r(65) = -0.304, p =.01} respectively).


\begin{figure}[!htp]
    \centering
    \includegraphics[scale=.3]{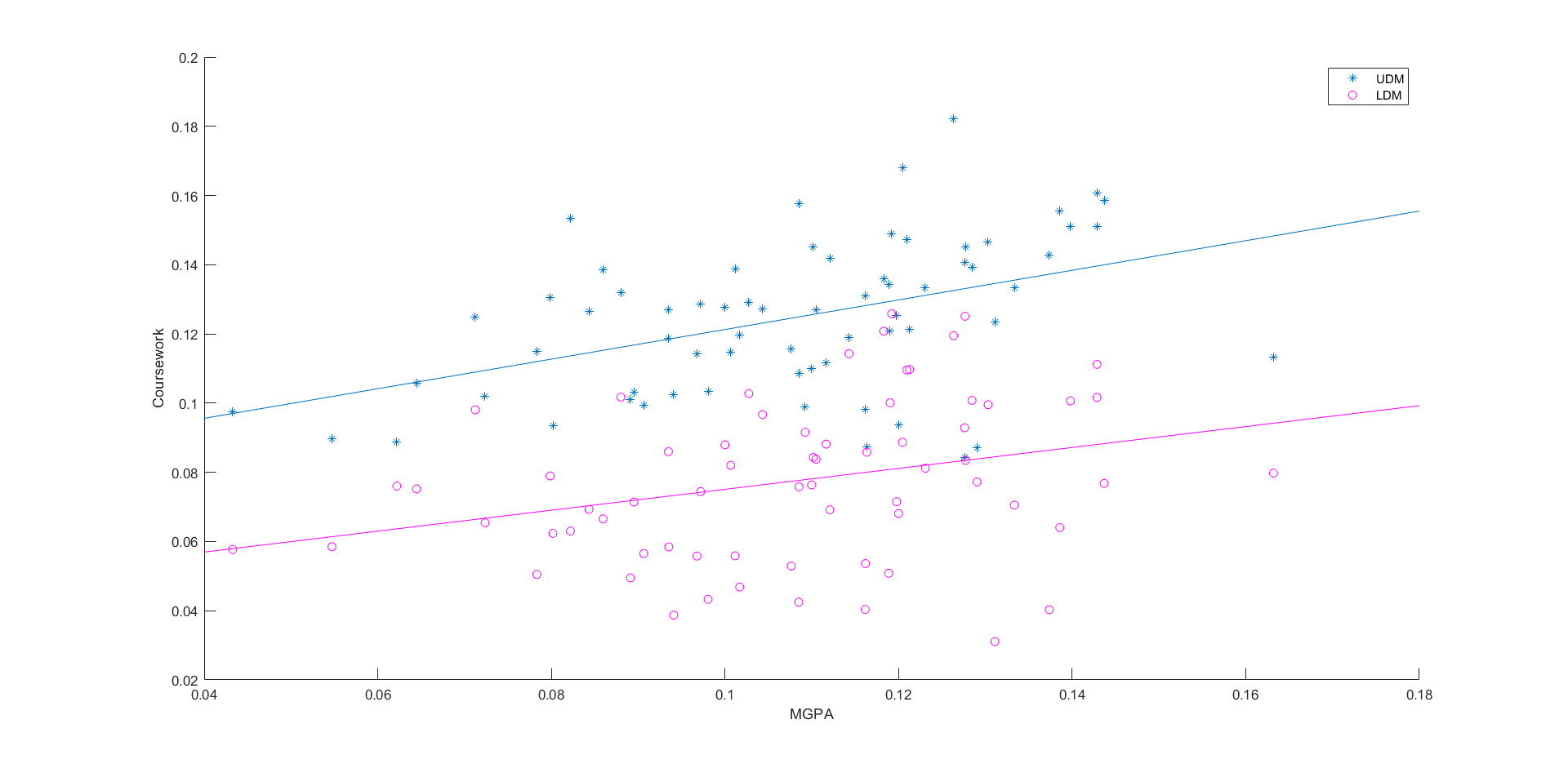}
    \caption{Correlation: Math GPA}
    \label{fig:Correlation: Math GPA}
\end{figure}

\subsubsection{Math Coursework}

Whereas upper division courses were found to be negatively correlated with the verbal GRE (\textit{r(65) = -0.412, p =.0005}), lower division courses were found to be negatively associated with both background (\textit{r(65) = -0.2619, p =.032}) and major (\textit{r(65) = -0.391, p =.001}). Upper division and lower division courses share a moderate positive correlation (\textit{r(65) = .328, p =.007}) with each other, meaning professors who tend to value upper division courses also tend to value lower division courses.

\begin{figure}[!htp]
    \centering
    \includegraphics[scale=.3]{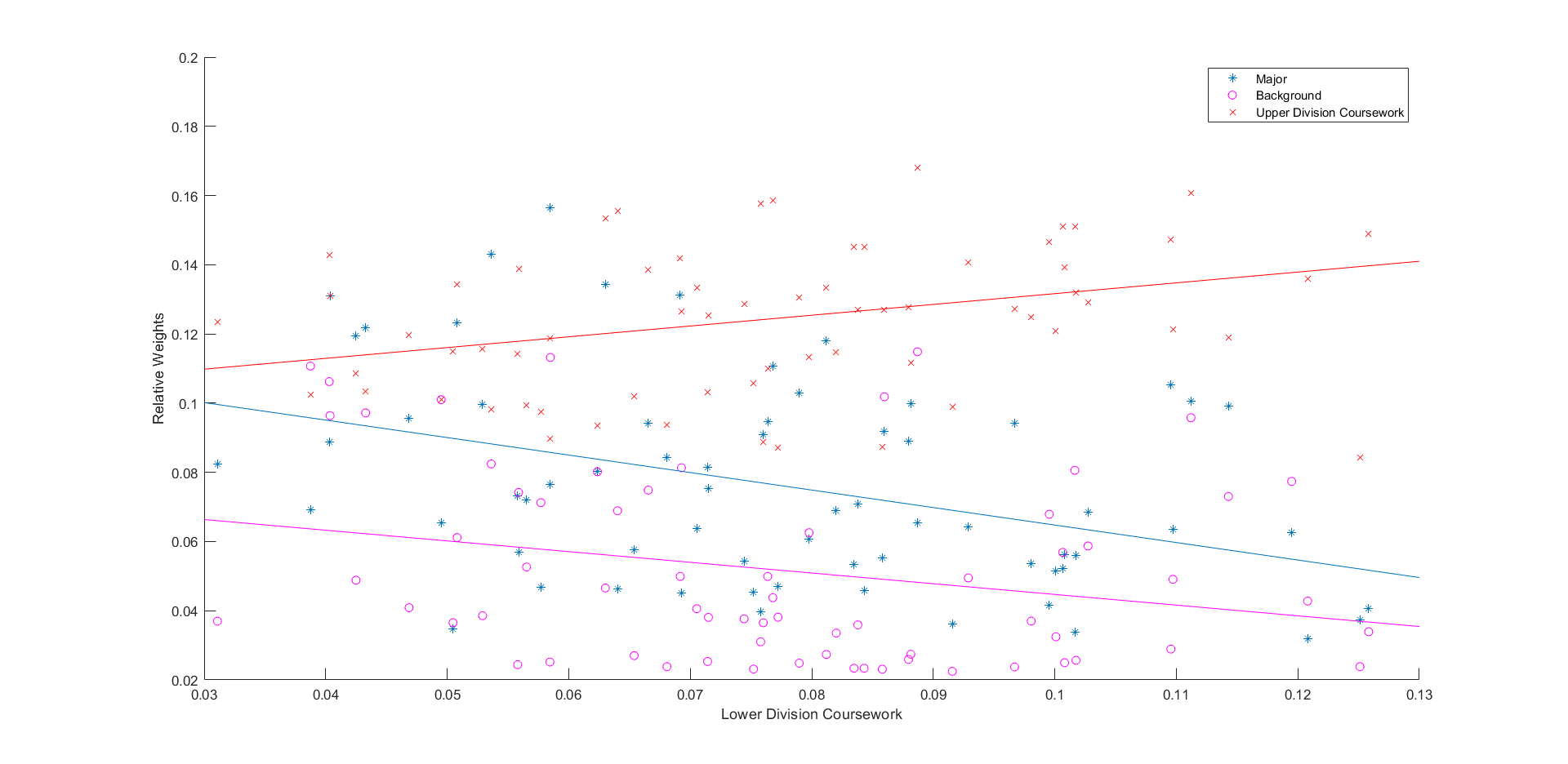}
    \caption{Correlation: Lower Division Coursework}
    \label{fig:Correlation: Lower Division Coursework}
\end{figure}




\subsubsection{Undergraduate School Tier}

Results indicate a positive correlation between an applicant's undergraduate school tier and upper division math grades (\textit{r(65) = -.256, p =.045}). That is, professors serving on graduate admission committees who place high value on an applicant's undergraduate school tier tend to place less emphasis on one's upper division math courses. Additionally, results also indicate an applicant's undergraduate school tier is also negatively correlated with  one's math GPA (\textit{r(65) = -.225, p =.067}) and and lower division math grades (\textit{r(65) = -.232, p =.059}). Again, although the last two correlations are not statistically significant at a $.05$ significance level, the p-value is nonetheless presented and left to the reader's discretion. 

\begin{figure}[!htp]
    \centering
    \includegraphics[scale=.3]{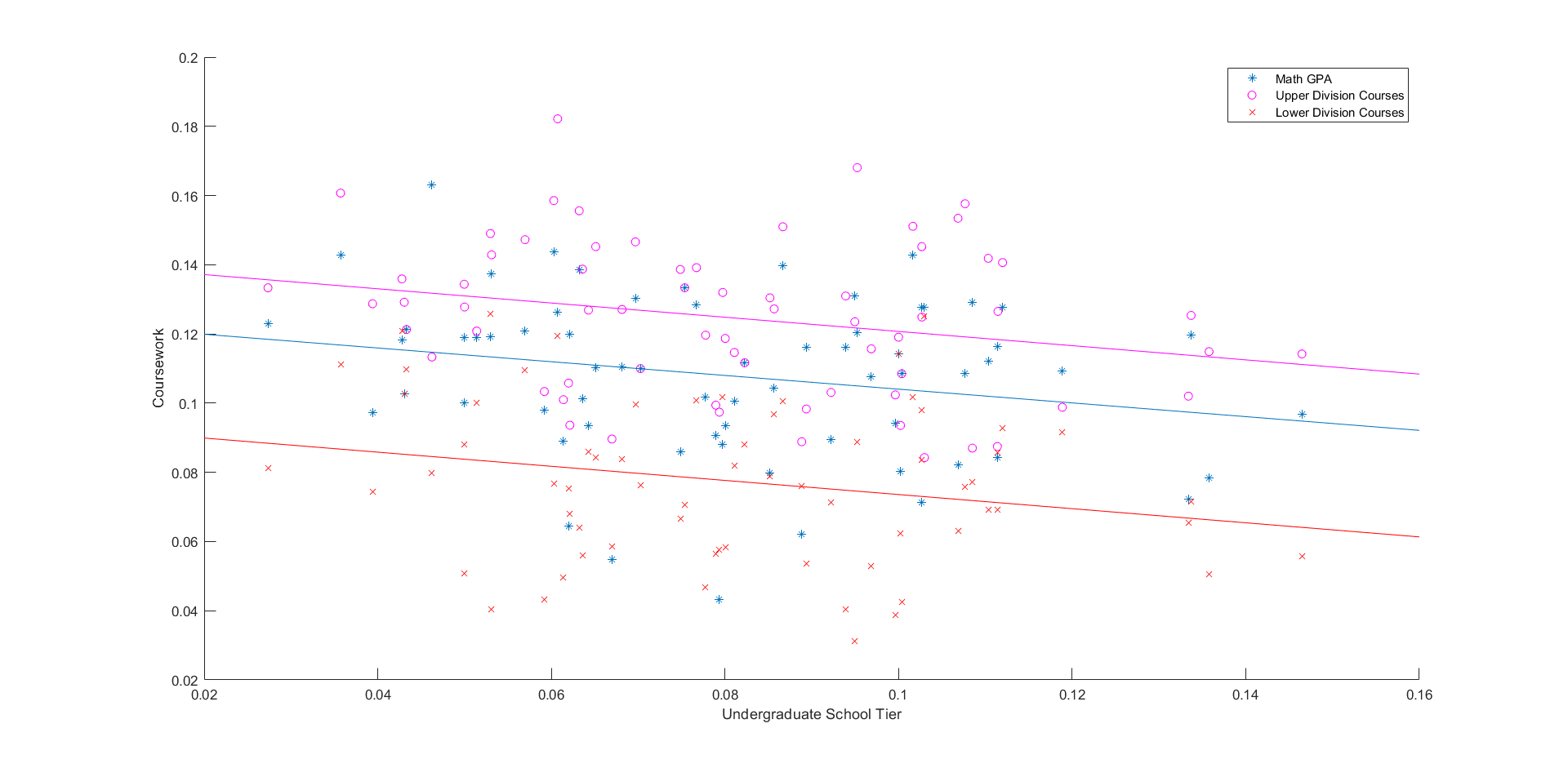}
    \caption{Correlation: Undergraduate School Tier}
    \label{fig:Correlation: Undergraduate School Tier}
\end{figure}

\section{Discussion}
The following section will provide insight on the results presented above. Additional commentary from professors who were interviewed will be provided. All professors quoted have served on graduate admission committees unless otherwise noted. The section is organized by our discussion of each variable. Boxplots are provided for visualizations with the red line indicating the mean and the upper and lower blue lines indicating the 75\% and 25\% percentiles respectively.

\subsection{Background}

At the beginning of the study, full professors were hypothesized to value background less than associate and assistant professors. However, as findings indicate there is no statistically significant differences between this population, it can be inferred that there is a greater consensus amongst professors on inclusive excellence than originally thought. As a result, this lack of evidence indicates that diversity does not vary based on professorship rank.

\begin{figure}[!htp]
    \centering
    \includegraphics[scale=.3]{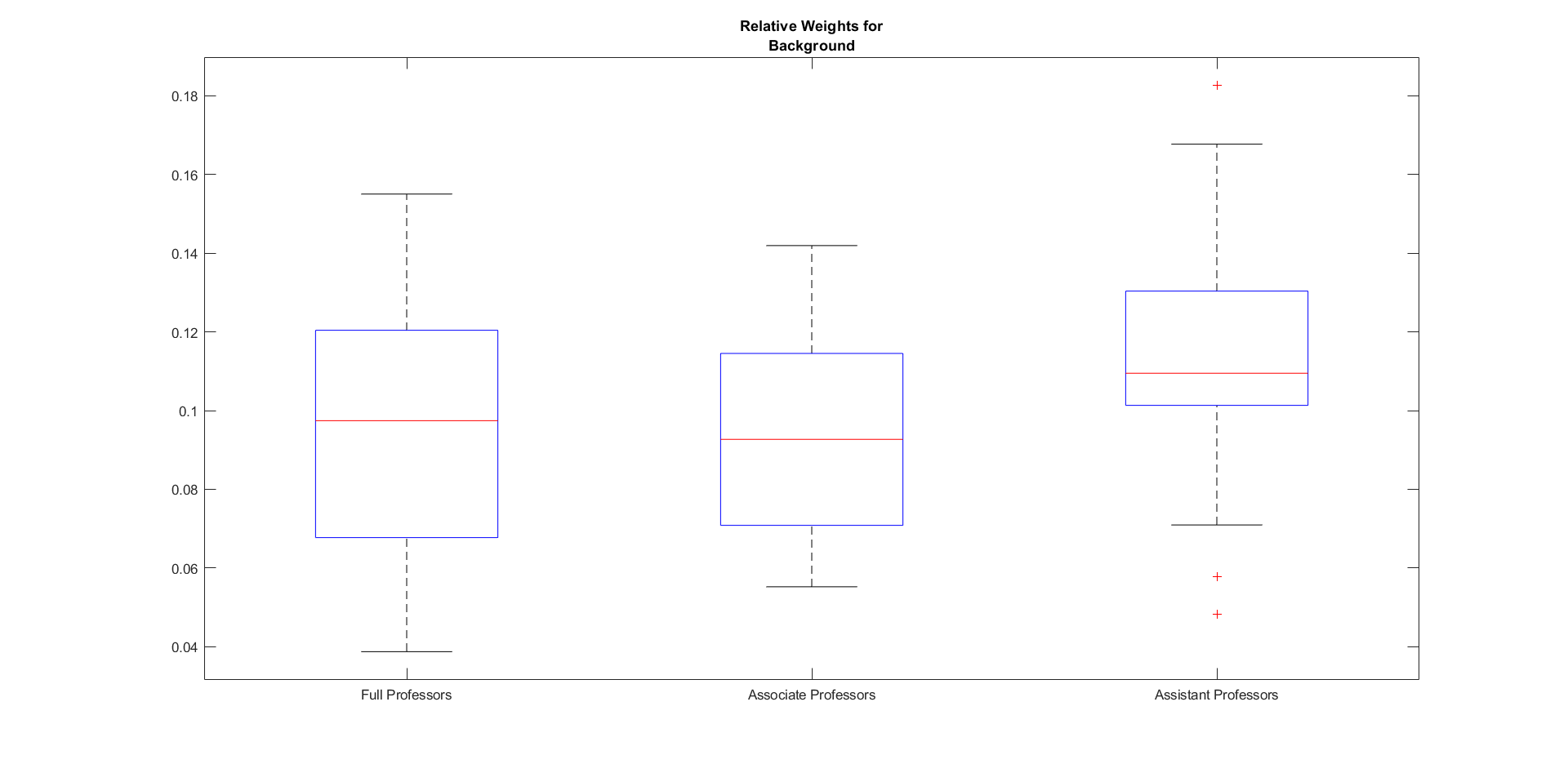}
    \caption{Boxplot: Background by Professorship Rank}
    \label{fig:Correlation: Background by Professorship Boxplot}
\end{figure}

However, a statistically significant difference does present itself when testing between school tiers. Rather than differing by professorship rank, it was shown that emphasis placed on an applicant's background actually differs by the ranking of the mathematics PhD program. It was shown that the top 30 schools value an applicant's background statistically less than schools outside the top 30. This not only implies that schools in the top 30 value other admission factors significantly more than background but also that minorities may have a more difficult time being admitted into a top 30 mathematics PhD program. Such findings present the issue of inclusive excellence in the top 30 mathematics PhD programs and brings into question how these schools are committed to promoting inclusive excellence. 

When asked, Associate Professor Nathan Kaplan of the University of California, Irvine responds that he makes an attempt "to be understanding of the challenges faced by and underrepresented minority applicants and to look for ways that we can use graduate admissions to make the department more inclusive."
Likewise, Assistant Professor Leslie New from Washington State University states that background provides an additional indicator of success and perseverance, with such attributes contributing to one's academic and research potential.

\begin{figure}[!htp]
    \centering
    \includegraphics[scale=.3]{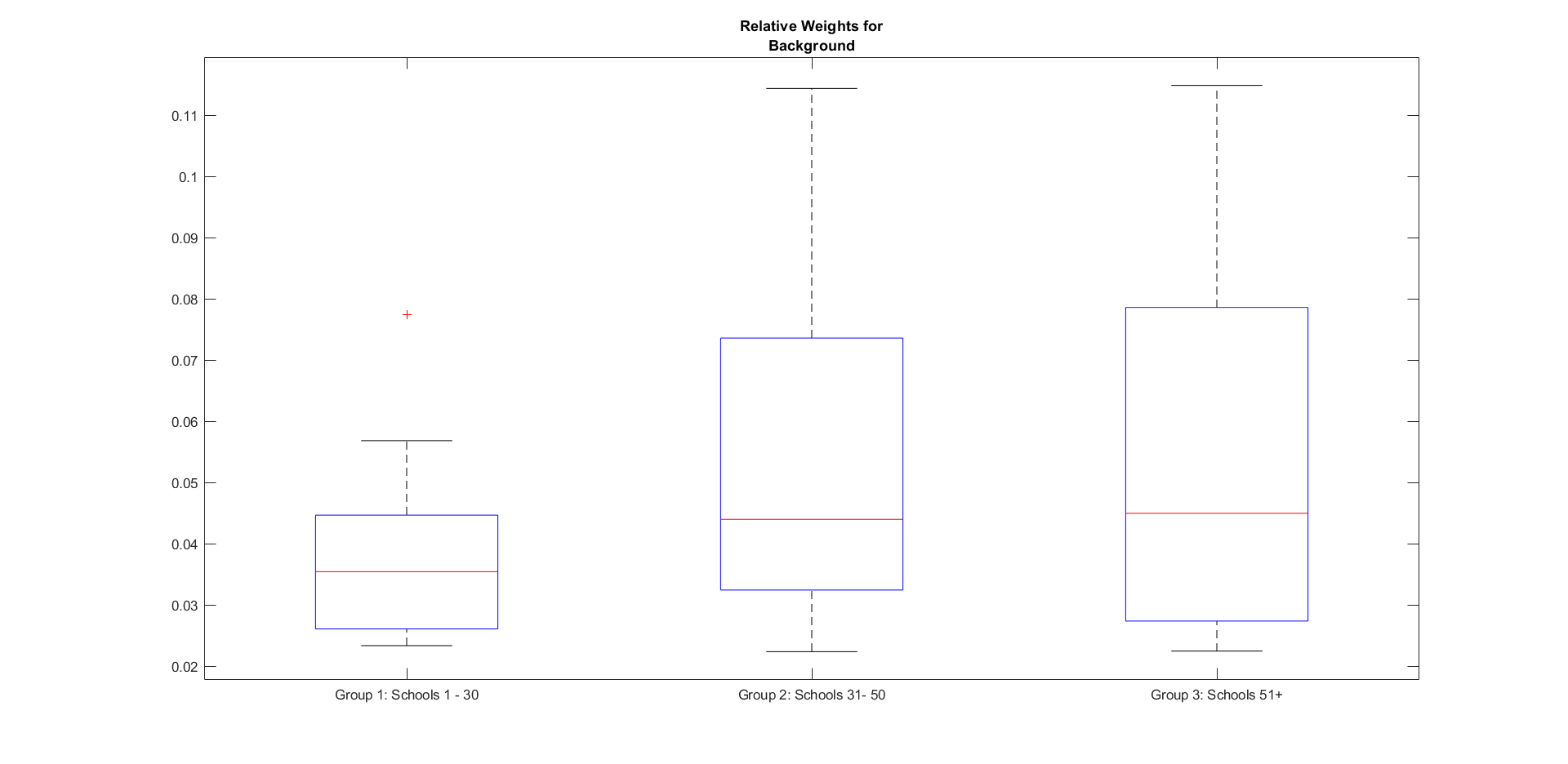}
    \caption{Boxplot: Correlation: Background by Groups}
    \label{fig:Correlation: Background by Groups Boxplot}
\end{figure}


The phenomenon that background is valued less at the top 30 mathematics PhD programs is not surprising as these programs are often receiving much more applications than they can accept. Graduate admission committees at these programs, then, have to rely on more traditional measures of academic excellence such as GPA and standardized test scores. As a result, selecting the best applicants by traditional measurements often overshadows diversity corrections in these top programs. 


\subsection{Standardized Testing}

\subsubsection{Background}

The topic of standardized testing and discrimination has long been a heated issue. Whereas some professors value standardized testing for its consistency and reliability, other professors consider them as an institutionalized barrier to higher education. Professor New goes so far as to say that "Standardized test metrics should be dropped completely and play no role in graduate admissions." Likewise, full professors Dieter Armbruster, associate director for graduate programs Arizona State University states that the GRE is "definitely biased against minorities" and "not indicative of math skills." 

However, some professors interviewed commented that the math subject GRE actually provides a more level playing field for applicants. Although he still recognizes the equity issues with standardized testing as a whole, Assistant Professor Alex Cloninger of the University of California, San Diego comments that the material on the mathematics subject exam is mostly covered in a standard mathematics curriculum, allowing students to have access to the exam material in their instruction. For such reasons, the subject exam allows less leeway for inequality issues.

The dichotomy between the two ideologies was clearly expressed in the findings. The negative correlation between background and standardized testing demonstrates that professors who ranked background higher tend to rank standardized testing lower, whereas professors who ranked background lower tended to place more weight on standardized testing. Additionally, all three GRE scores were positively correlated with each other. That is, they were in a sense a one deal package. This means that if a professor valued one aspect of the GRE score, he is also more likely to value the other two GRE scores, whereas a professor who disregards one aspect of the GRE is also most likely to disregard the other two aspects. The only exception to this is the relationship between the verbal GRE and the math subject GRE, in which there is no statistically significant relationship. A visual representation that exemplifies these relationships can be found in Figure \ref{fig:Background and Standardized Testing Correlation Plot}. With these statistics, we conclude that professors who value one's background tend to devalue all three metrics of the GRE whereas professors who value one's background less tend to value all three metrics of the GRE more.

\begin{figure}[!htp]
    \centering
    \includegraphics[scale=.3]{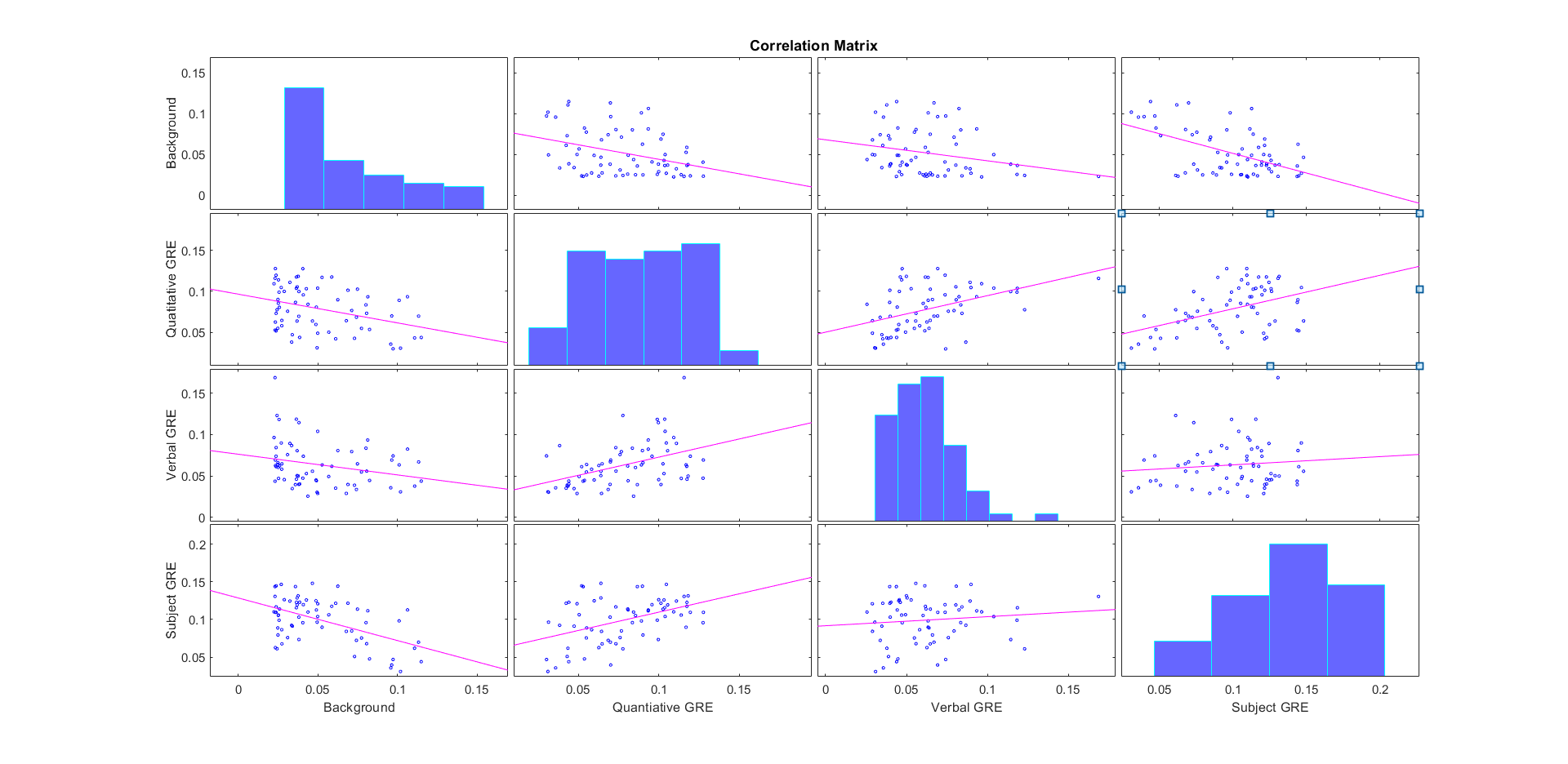}
    \caption{Background and Standardized Testing Correlation Plot}
    \label{fig:Background and Standardized Testing Correlation Plot}
\end{figure}

\subsubsection{Mathematics Subject GRE}
The analysis indicates that the math subject GRE is valued much more for programs in the top 30 rankings whereas there are no differences between programs outside of the top 30. This indicates that applicants applying to the top 30 math PhD programs should give more time to the subject GRE than those who are applying to schools outside the top 30. However, applicants only considering programs outside the top 30 can dedicate their time on other factors as their math subject GRE score will be given a lower weight. Additionally, applicants who are debating between applying to schools both in the top 30 and outside the top 30 should be aware that their subject GRE score will make a larger difference in their application to the top 30 but not much of a difference between schools outside the top 30.

\begin{figure}[!htp]
    \centering
    \includegraphics[scale=.3]{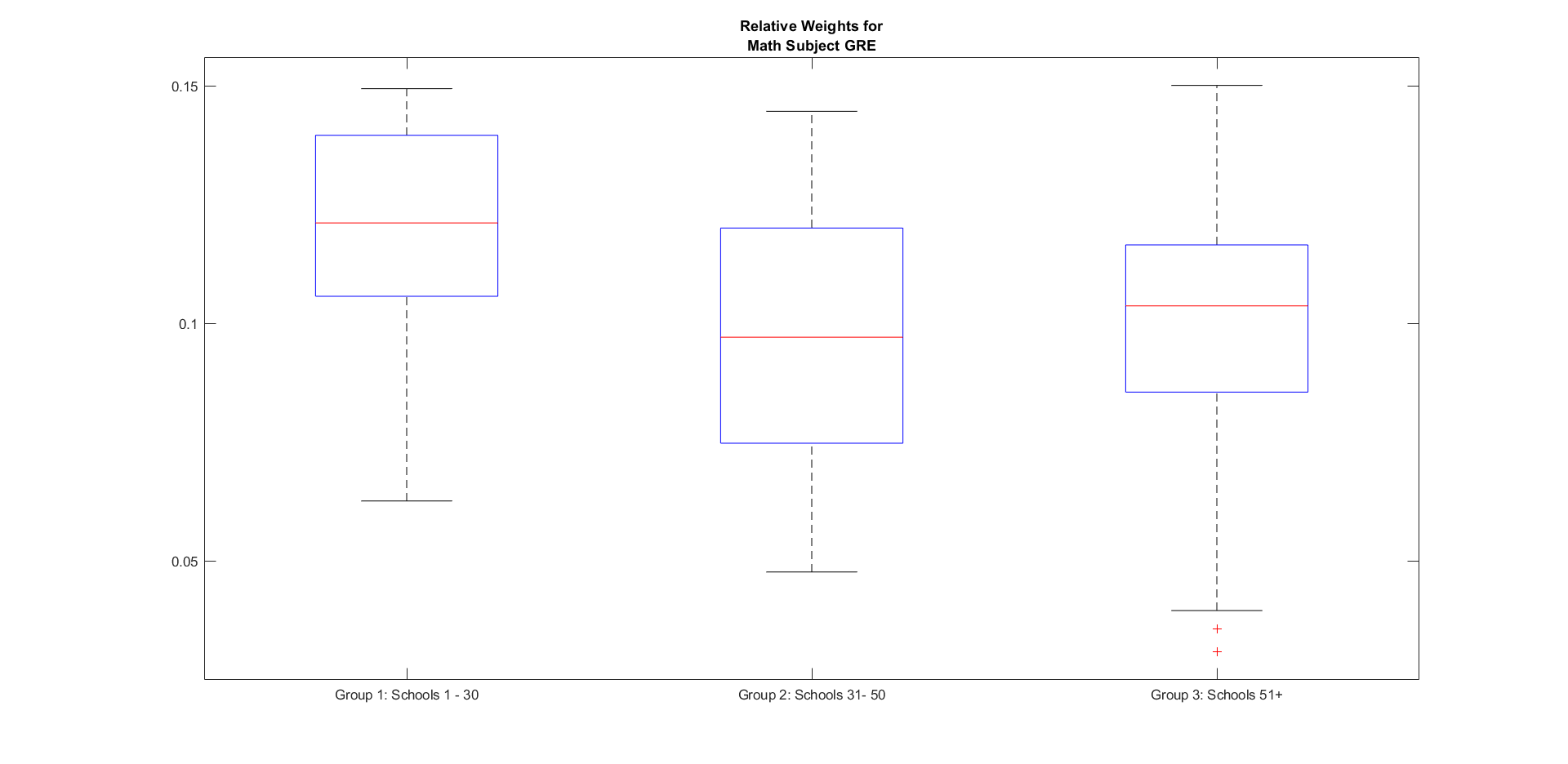}
    \caption{Boxplot: Mathematics Subject GRE by Groups}
    \label{fig:Mathematics Subject GRE by Groups}
\end{figure}

\subsubsection{Quantitative GRE}

Full Professor Elizabeth Meckes at Case Western Reserve University comments that the verbal GRE score was once valued slightly higher 30 - 40 years ago as graduate admission committees had less data to evaluate a candidate on. Advanced mathematics courses were not as available for undergraduates, the GRE subject exam was only offered later, and undergraduate research was not as prevalent. With less signals to evaluate a candidate on, graduate admission committees had to place more weight on other factors to differentiate candidates. As a result, standardized test scores were given more weight due to the lack of other signals.

Older full professors, then, may still carry on this emphasis. One anonymous professor describes the continuing emphasis on the quantitative GRE as similar to an academic "inertia." Anecdotally, he perceived that the test scores were definitely a "bigger deal than it was today." Full professors, then, who placed much emphasis on standardized test scores years ago may carry on this emphasis that younger associate and assistant professors do not. In the words of Professor Meckes, "Once professors set their standards, they carry on that standard and may not reevaluate them." Such reasoning can also be applied to the math subject GRE score in which there is a small but not statistically significant difference.

\begin{figure}[!htp]
    \centering
    \includegraphics[scale=.3]{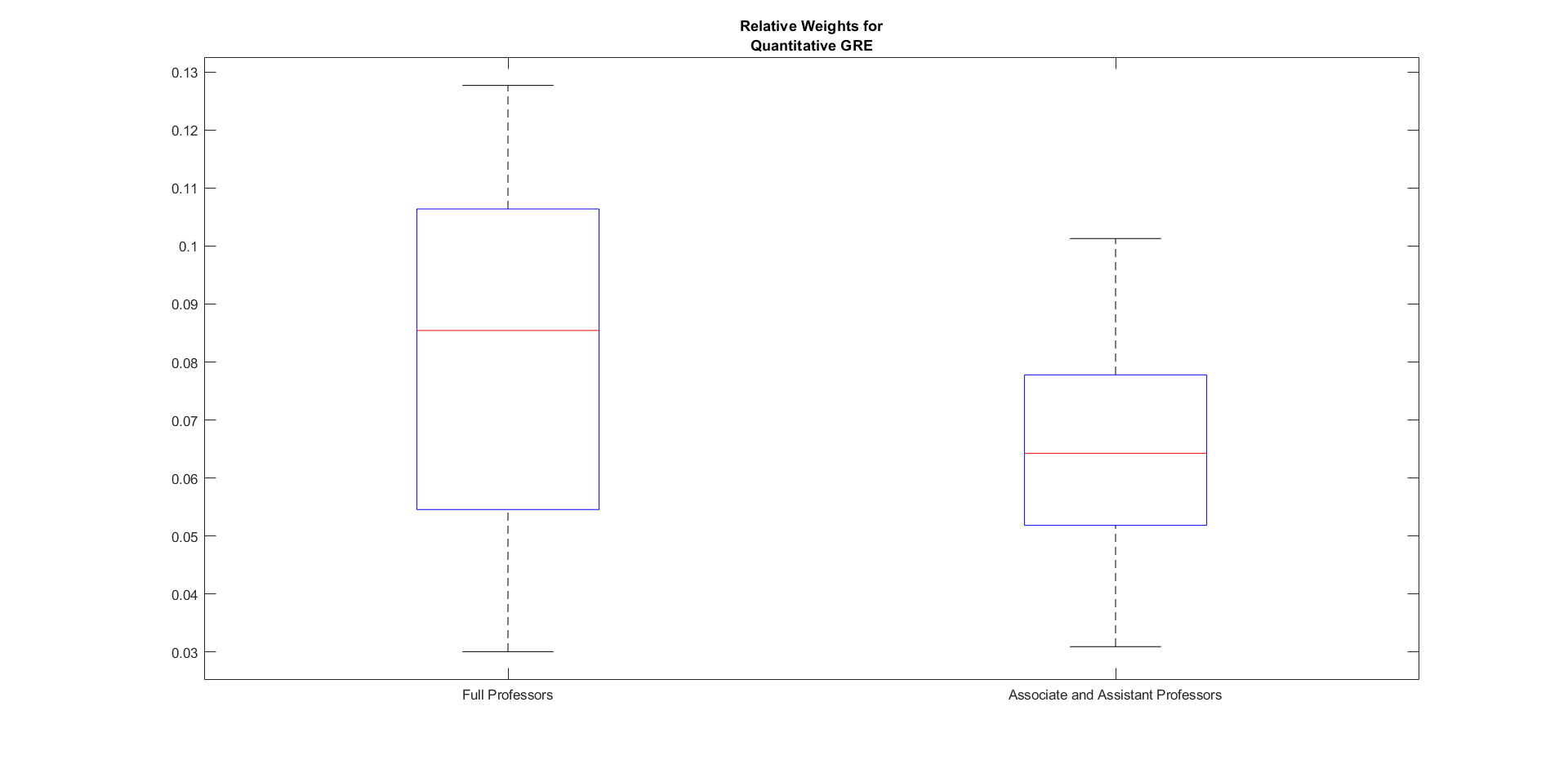}
    \caption{Boxplot: Quantitative GRE by Professorship Rank}
    \label{fig:Quantitative GRE by Professor}
\end{figure}

\begin{figure}[!htp]
    \centering
    \includegraphics[scale=.3]{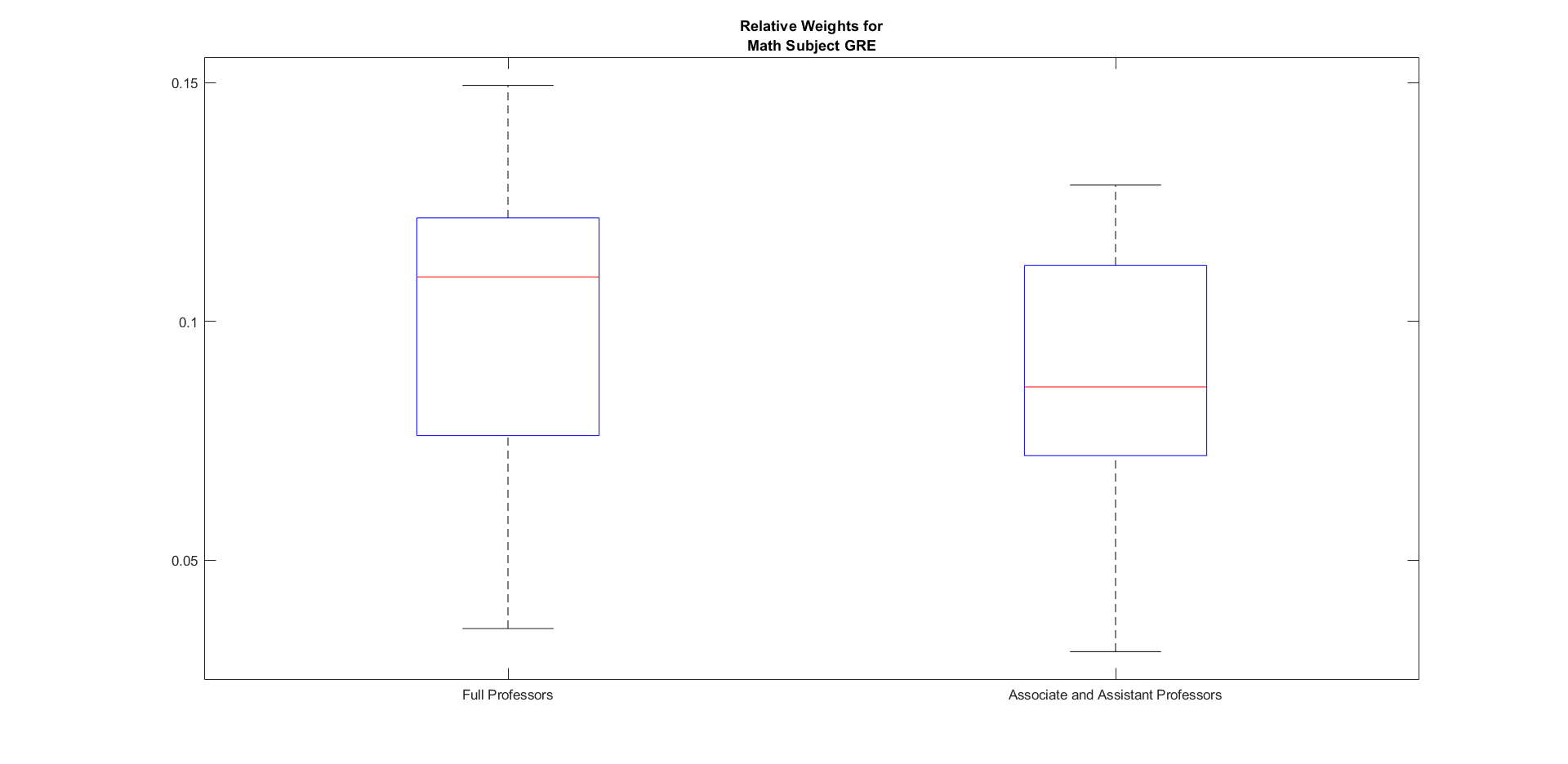}
    \caption{Boxplot: Mathematics Subject GRE by Professorship Rank}
    \label{fig:Mathematics Subject GRE by Professor}
\end{figure}

\subsubsection{The Role of the Verbal GRE Score}


As more signals for academics and research potential became available, the verbal GRE has become less relevant when evaluating a candidate. It is not surprising, then, that the verbal GRE was found to be negatively correlated with almost every other variable. However, this is opposed to the findings of Timmy Ma and Karen E. Wood \citep{Ma}. As described in the literature review, Ma and Wood find that the verbal score is positively correlated with graduating from a mathematics PhD program, even moreso a stronger predictor as there is little variation in the quantitative GRE scores of PhD candidates. The opposition of our findings with Ma and Wood's suggests that graduate admission committees are unaware of the usefulness of the verbal GRE scores in differentiating candidates. As the quantitative GRE is not able to differentiate candidates due to its little variation in scores, the verbal score may provide an additional way of evaluating candidates. 

Despite these findings, it is clear from our data that the verbal GRE is not valued during an admission's process as shown by its lowest ranking placement in Table \ref{tab:Served on Graduate Committee and All Factors}. Although applicants seeking to optimize their admission chances should not focus on the verbal GRE scores due to its low rankings, graduate admissions committee may find it worthwhile to look further into the phenomenon of why positive verbal GRE scores are correlated with successful PhD candidates.
\subsection{Undergraduate School Tier}

Results indicate that the top 50 schools give more weight to an applicant's undergraduate institution. This is due to the fact that professors who are familiar with an applicant's undergraduate institution will have a better understanding of the math curriculum at that institution, allowing them to better judge the applicant. Conversely, if graduate admission members don't recognize an applicant's school, they may not be able to have a good understanding of how rigorous the math curriculum is at that institution. Professor Meckes' comments perhaps clarifies the value of tier best: "This isn't just snobbery. It's another data point that provides context."

\begin{figure}[!htp]
    \centering
    \includegraphics[scale=.3]{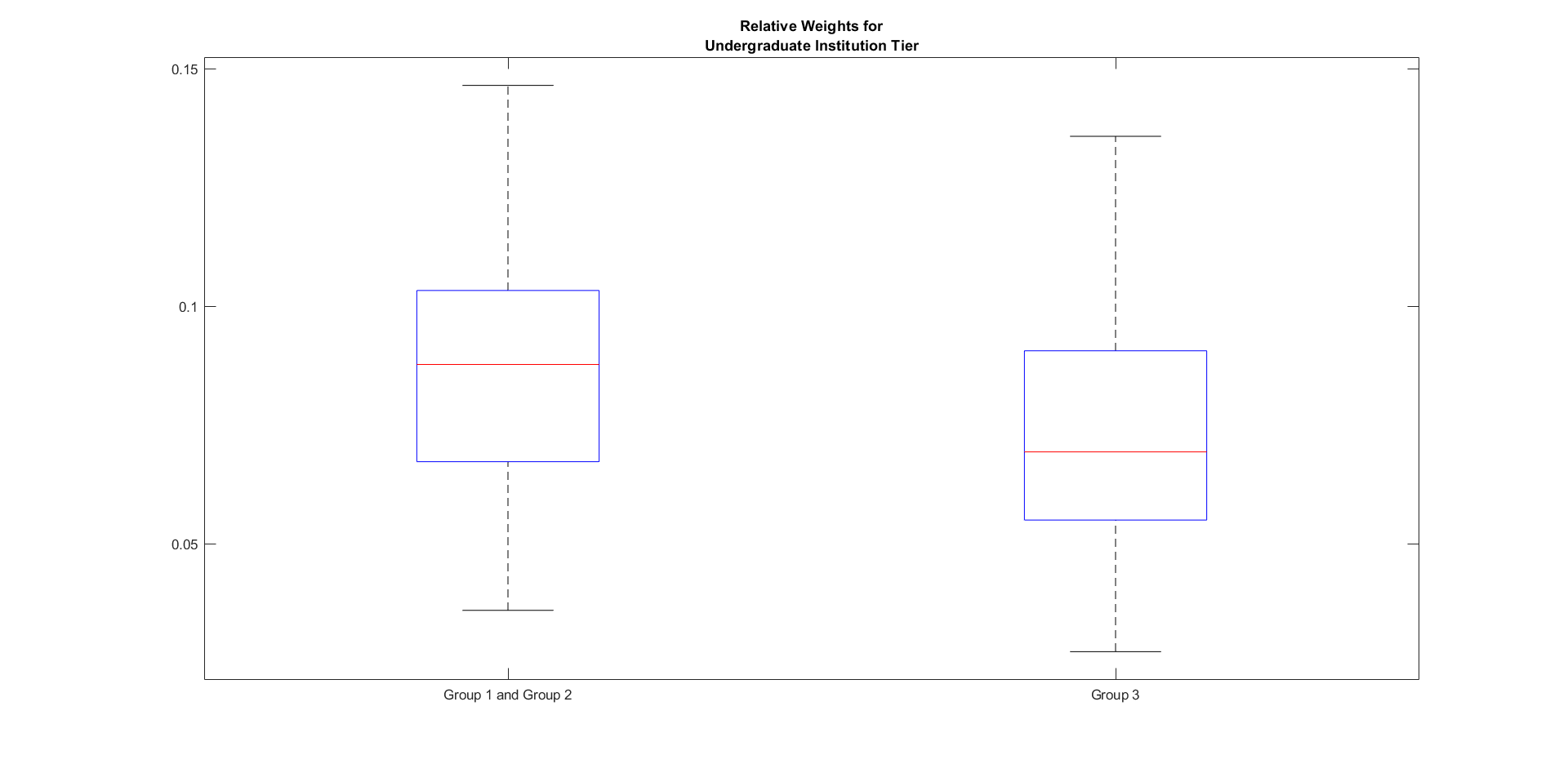}
    \caption{Boxplot: Undergraduate School Tier by Groups}
    \label{fig:Undergraduate School Tier by Groups}
\end{figure}

This may suggests a sort of "academic inequality" between school tiers. The fact that the top 50 schools emphasize an applicant's undergraduate school tier more implies that applicants from lower schools will have a difficult time climbing the "academic ladder." As a result, applicants from smaller schools will have to provide more evidence of mathematical competence by receiving higher grades in more rigorous courses to signal the same academic competence as their peers from more well-known institutions.

However, it is important to note that applicants from lower ranked schools have many opportunities to provide evidence of academic competence as graduate admission committees view one's application comprehensively. Professor Kaplan states that applicants are evaluated relative to their institution and that applicants who come from smaller colleges with limited programs but have taken the highest level courses offered are viewed very positively. Additionally, although he has not served on a graduate admissions committee, Visiting Assistant Professor Henry Tucker of the University of California, Riverside shares a similar opinion stating that "I would choose a student from an unknown school who took the hardest courses offered over a student from a highly-ranked school who only did the absolute minimum requirements to graduate." 

Tucker continues to state that although he acknowledges the obstacles standardized testing presents to minorities, it can also be an opportunity for applicants from smaller schools. Standardized testing, he says, actually provides an opportunity for students at smaller schools with less academic opportunities to demonstrate their mathematical proficiency and competence on the same level as students from larger schools. An anonymous professor who had previously served as vice-chair of graduate admissions at a tier one institution agrees with this sentiment stating that disadvantaged or minority students with high math subject GRE scores whose application may otherwise appear unassuming definitely stand out to his admissions committee. As a result, since the math subject GRE is a standardized and reliable method of evaluating all test takers, applicants from lesser-known schools should especially seek to do well on this exam.

The fact that the top 50 math PhD programs prefer applicants from higher ranked schools coincides with the negative relationship undergraduate school tier shares with math GPA, upper division math, and lower division math. This means that those who value academic coursework more value an applicant's undergraduate school tier less, whereas those who value coursework less tend to value the applicant's undergraduate school tier more. This is due to the fact that the rigor of one's coursework often varies by school. One anonymous lecturer at the University of California, Berkeley, comments that "Two courses from two different schools may have the same name and same design yet the level of difficulty may be completely different." This finding, then, reflects Ma and Wood's results that one's undergraduate tier is more associated with success in a mathematics PhD program than GPA due to higher rigor at top schools.

As a result, professors who review an applicant from a highly ranked school may be less concerned with lower grades due to the academic rigor associated with the school. Conversely, a professor who sees an applicant from an institution he doesn't recognize may not have a good understanding of the institution's academic rigor and therefore may not know how to weigh the applicant's coursework. As a result, the professor may be less lenient in viewing the application and expect higher grades in coursework to provide a stronger signal of academic proficiency. As PhD applicants from higher ranked schools are given more leniency on coursework due to the perceived higher rigor at more well-known institutions, this buttresses the idea that it may be difficult to climb the "academic ladder." 

In addition to doing well on standardized testing, PhD applicants can plan their coursework to present a comprehensive application by taking courses that relates to one's research interests. Although he has not served on a graduate admissions committee, Assistant Professor Daniel Conus of Lehigh University states that "I personally view very favorably students who already have a sense of what area of mathematics they would like to specialize in (as much as can be expected) and who contact specific faculty members to get an idea of what is available to them." For instance, an applicant who intends to study applied mathematics in a PhD program can create a comprehensive application by taking extensive coursework in numerical analysis, optimization, and PDEs. Achieving high scores in such classes will not only provide a positive signal of academic proficiency but also demonstrate a genuine research interest. Additionally, having research experience that complements one's coursework and research interest will not only create a comprehensive application but also provide a strong signal of both genuine research interest and research potential. Full Professor Sergey Lapin of Washington State University states that "It is very important to extend one's curiosity and knowledge past a classroom setting. Applicants should be asking deeper questions such as 'What's the point? Why do we learn this? Why and how does this work?' Students who don't ask these questions are not accepted into graduate programs." Having both coursework and related research, then, will not only provide a strong indicator of academic curiosity but also signal to graduate admission committees the seriousness of one's academic goals.

\subsection{Math Coursework}

It is not surprising that associate professors, assistant professor, visiting assistant professors, and lecturers value lower division courses more than their full professor counterparts. An obvious explanation is due to the fact that the former group is more likely to teach lower division courses than the latter. The same reasoning explains why there is an even larger gap between lecturers and full professors in regards to lower division courses, especially as lecturers often do not teach upper division courses.

\begin{figure}[!htp]
    \centering
    \includegraphics[scale=.3]{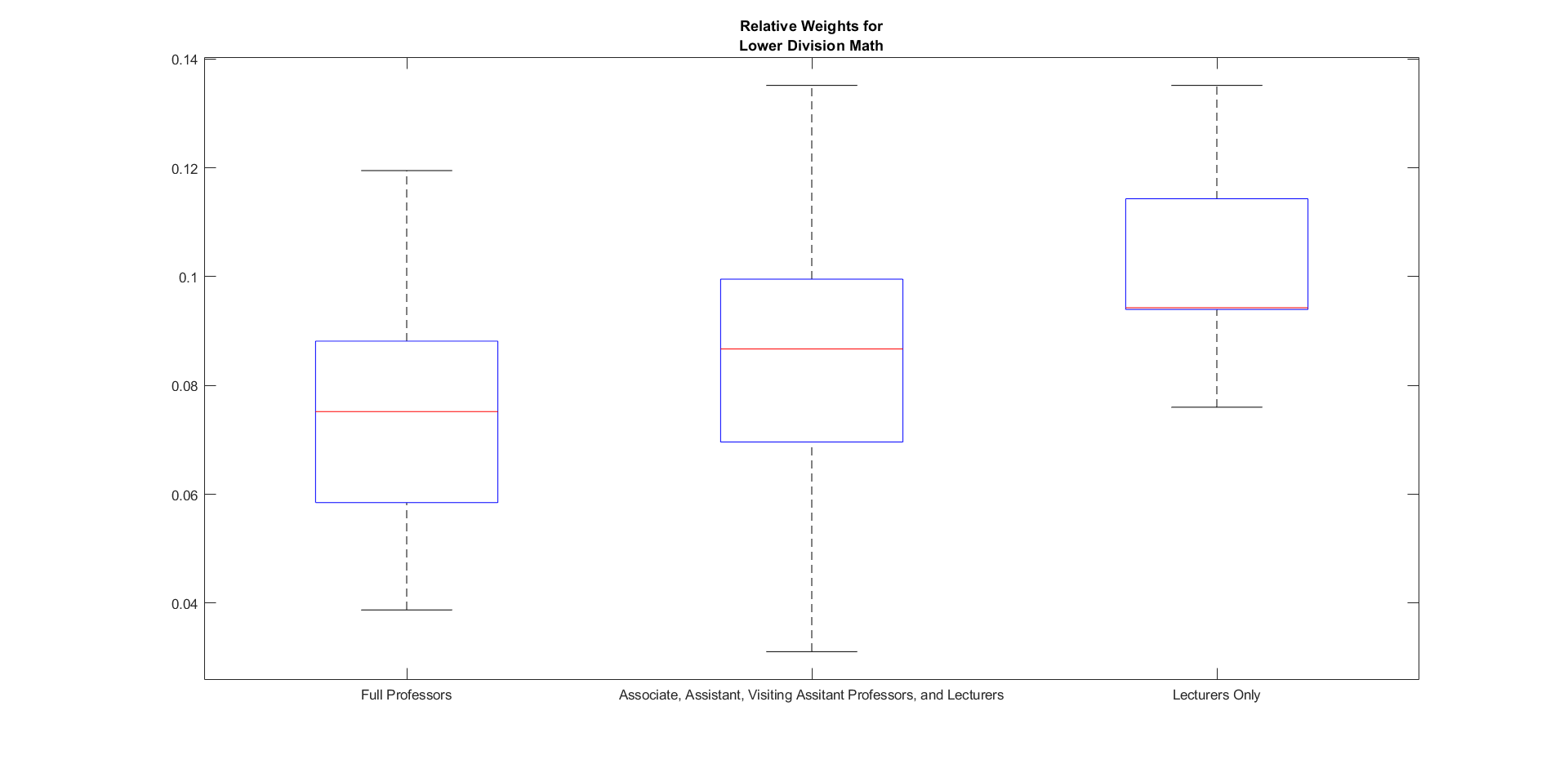}
    \caption{Boxplot: Lower Division Courses by Professorship Rank}
    \label{fig:Lower Division Courses by Professorship}
\end{figure}

However, it is surprising that lower division and upper division coursework share a positive relationship. Originally, it was expected that the two share a negative relationship. That is, professors who value upper division coursework were expected to not value lower division coursework as the former overshadows and supersedes the latter. Yet, upper division coursework was shown to be positively correlated with lower division coursework. This means that professors who valued upper division coursework highly also valued lower division coursework highly.

One explanation comes from the insight of Anton Gorodetski, Full Professor at the University of California, Irvine. He states that whereas lower division courses provide students with the basic mathematical tools to solve a standard set of problems, upper division courses provide the mathematical reasoning and theory on why and how these tools work. As a result, a student's understanding of the tools used at a lower division level is fulfilled with the theory learned at the upper division level. For instance, a first year calculus student's understanding of continuous functions is only complete when learning the standard $ \epsilon-\delta$ definition. As such, the positive correlation between upper and lower division courses reflects the deeper level of understanding professors expect students to have when transitioning from lower division to upper division coursework.

Lower division courses were also shown to be negatively correlated with a candidate's major. That is, professors who place more value on a candidate's major tend to value lower division courses less, whereas professors who place less value on a candidate's major tend to value lower division courses more. One reasoning for this negative correlation is that math PhD candidates who did not major in mathematics are expected to have the same mathematics skills as other candidates. As a result, professors who overlook an applicant's major outside of mathematics will place more emphasis on lower division course grades. However, candidates applying for a math PhD who were math majors have taken more rigorous mathematics courses, therefore supplying a stronger indicator of mathematical competence than lower division courses. As such, this finding suggests that those applying to a math PhD program from different fields pay additional attention to their grades in lower division math courses.

\subsection{Research Experience}

It may strike as surprising to find that research experience was valued by assistant professors more than full professors. However, when considering the recency of undergraduate mathematics research as described by Gallian \citep{Gallian}, the results become more evident. As support for undergraduate mathematics research didn't become common until the late 1980s, many older faculty members did not have the same undergraduate research opportunities as their younger associate and assistant colleagues did. As time progressed, research opportunities for undergraduates became more popular, allowing more generations of mathematicians to be involved in research. As a result, associate and assistant professors were more likely to have been involved in research as undergraduates than older full professors, and therefore value undergraduate research experience to a higher degree. The recent burgeoning of undergraduate research in mathematics, then, explains the differences shown in Figure \ref{fig:Undergraduate Research Experience by Professorship}.

\begin{figure}[!htp]
    \centering
    \includegraphics[scale=.3]{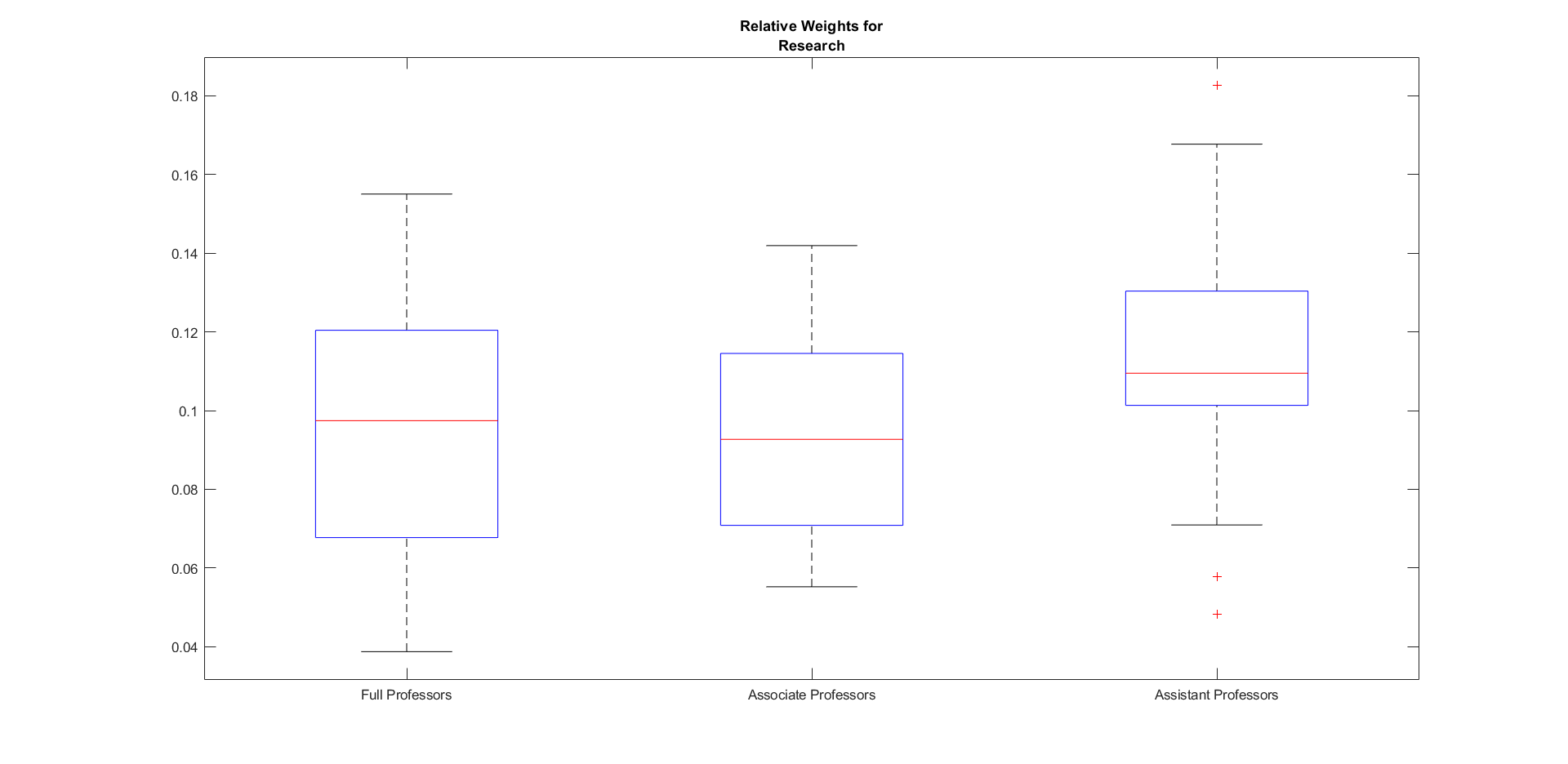}
    \caption{Boxplot: Undergraduate Research Experience by Professorship Rank}
    \label{fig:Undergraduate Research Experience by Professorship}
\end{figure}

Reflecting the same skepticism mathematicians had when undergraduate research in mathematics first began, Full Professor Stephen Bigelow of the University of California, Santa Barbara states that "'research experience' is more important than it should be because an undergraduate is not normally ready for research." Professor Meckes also describes that this common aggrandizement of undergraduate research results from the fact that research experience is both oversold by universities and overbought by students.  She expresses that the importance placed on undergraduate research is overemphasized when in reality such experience is not the "end-all be-all." This is especially true as more students are placing research experience on their CVs despite performing menial work. As such, having "research experience" has became an inflated term that has an ambiguous meaning. 

One anonymous professor who had served as vice-chair for graduate affairs at a tier one institution states that he was never sure how to interpret undergraduate research experience due to this ambiguity and that just participating in an REU is not going to make a big difference unless it's supported by an objective endorsement by a credible third party. In fact, several interviewed professors have stated that it is considered a major red flag if they see an applicant who participating in a research experience but did not receive a recommendation letter from their advisor - doing so being comparable to a PhD candidate not being recommended by his thesis advisor. Such skepticism, though not necessary as intense, resembles the original opposition to undergraduate math research when first established in the 1980s.

As such, the true value of research is not intrinsically in the research experience but rather in the letter of recommendation. This is due to the fact that since research experience is often an inflated term, admission committees need an objective means of viewing such experiences. Professor Meckes even states that she automatically disregards research experience if it doesn't come with a letter of recommendation to verify the nature of the research and the student's role in it. Professor Armbruster likewise states that the type of research program, whether it's a publication, REU, independent study, or honors thesis does not matter. What matters most to him is the insight given in the letter of recommendation. Letters from advisors and program directors should therefore address the nature of the research project and the student's role in the project. Professor Cloninger lists the following qualities that his admission committee would like to see described in a candidate's research experience: "initiative, independence, curiosity, perseverance, engagement, and problem solving." 

Our study's interviewees have additionally stated that the results of an applicant's research experience doesn't necessarily have to be publishable or revolutionary. In fact, one professor even comically described the chances of an undergraduate publishing novel work is "less than $\epsilon$." An anonymous professor states that this is because math research is comprehensive in nature and requires the full understanding of the past. However, Professor Lapin states that results should nonetheless be presentable. Official presentations at a conference allows graduate admission committees to know that the applicant was not only able to conduct research but also that the research was suitable to be presented publicly to others. Such presentations give weight to an applicant's research experience by demonstrating that the research was at the level of public presentation and therefore provides an additional objective signal of research potential.

As a result, if applicants are planning to engage in research, applicants must not only demonstrate the qualities described by Professor Cloninger but more importantly provide a letter of recommendation that affirms and vouches such qualities. Without a recommendation describing the nature of one's work, research experience may not be given any weight by graduate admission committees. It is conclusive, then, that the value of undergraduate research in regards to graduate admissions lies not necessarily in the research experience itself but more so in the advisor's recommendation letter.

\subsection{Limitations}

It is important to note that the variables above are not comprehensive. There are many other factors to consider in a candidate's application, most notably letters of recommendations. However, letters of recommendations were excluded as students do not have full control over this factor. As a result, the results from the AHP analysis will only present the relative significance of the variables given above.

This paper also acknowledges the potential of a Type I error in its analysis. This is due to the increase in family-wise error rates (FWER) as multiple correlation hypothesis tests between all 12 variables for a total of 144 correlation hypothesis tests were performed. However, the associated p-values have been provided for each test and the conclusions are left to the reader's discretion.

\section{Conclusion}
It is my hope that this paper's findings has shed some light on which criteria graduate admissions committees value most and that the subsequent discussion has explained the reasoning behind this ranking. Of course, this paper's findings does not comprehensively illuminate the entire process of graduate applications, but its analysis and subsequent discussion has perhaps shed enough light both to offer guidance to an undergraduate reader and to provide insight to professional academics serving on admissions committees. 
Upon finishing this article, the former will hopefully have realized where he should best dedicate his time preparing for graduate school applications, and the latter can likewise reassess their own admission criteria to either conform or disagree with the general criteria rankings presented here. 

Additionally, the results of this paper have provided not only a hierarchy of graduate admissions criteria but also a deeper insight into the values of the mathematics community. Philosopher Thomas Aquinas once wrote that "The things that we love tell us who we are." By looking at the admission criteria rankings, the reader is able to see what leading mathematicians "love," thereby gaining insight into the current priorities of mathematicians at the highest level of academia. In this sense, the ranking of graduate admissions criteria also reflects the values of the larger mathematics community.

Of course, many will argue for a change in these values. Prime examples of ever-changing criteria are the the use of standardized testing, the worth of research experience, and the approach to improving diversity. As new generations of mathematicians arise, the rankings may perhaps shift. Such was the case with the rising emphasis on undergraduate research experience and conversely with the decreasing relevance of the verbal GRE score. Which values the mathematics community wants to emphasize in the future, then, is perhaps unforeseeable. However, as the reader finishes this article, I ultimately leave it up to the them to decide what factors the future mathematics community deems important enough to value in graduate admissions.

\newpage


\bibliography{bibliography}
\newpage

\section{Appendix}

\subsection{School Rankings}\
\begin{table}[!htp]
\setlength\tabcolsep{1.5pt}
\def\arraystretch{.6}
    \centering
    \begin{tabular}{l|m{3cm}|m{3cm}|c}
    Institution & US World News Ranking & Group Categorization & Responses\\
    \hline
    Arizona State University&47&2&2\\
    Auburn University&94&3&1\\
    Binghamton University&86&3&1\\
    Bowling Green State University&N/A&3&1\\
    Brandeis University&47&2&1\\
    Brigham Young University&86&3&4\\
    Bryn Mawr College&N/A&3&1\\
    Case Western Reserve University&94&3&2\\
    Central Michigan University&N/A&3&1\\
    Colorado State University&74&3&1\\
    George Washington University&101&3&2\\
    Georgia Institute of Technology&26&1&1\\
    Indiana University, Bloomington&34&2&2\\
    Lehigh University&101&3&2\\
    New Mexico State University&117&3&1\\
    Portland State University&N/A&3&1\\
    University of California, Berkeley&2&1&6\\
    University of California, Davis&34&2&2\\
    University of California, Irvine&39&2&8\\
    University of California, Los Angeles&7&1&1\\
    University of California, Merced&N/A&3&1\\
    University of California, Riverside&71&3&5\\
    University of California, San Diego&19&1&4\\
    University of California, Santa Barbara&39&2&2\\
    University of California, Santa Cruz&71&3&2\\
    University of Arizona&47&2&6\\
    University of Nevada, Reno&N/A&3&1\\
    University of Oregon&55&3&3\\
    University of Pittsburgh&55&3&1\\
    University of Texas, Arlington&136&3&1\\
    University of Virginia&47&2&1\\
    University of Washington&26&1&4\\
    Washington State University&101&3&4\\
    Blank &N/A&N/A&32\\
    \hline
    Sum & N/A & N/A & 108
    \end{tabular}
    \caption{School Categorization}
    \label{tab:School Categorization}
\end{table}

\subsection{Methodology Scale}
\begin{table}
    \centering
    \begin{tabular}{c|l}
        Score & Meaning\\
        \hline
        1/3 & Strongly less important\\
        1/2 & Moderately less important\\
        1 & Similarly as important as\\
        2 & Moderately more important than\\
        3 & Strongly more important than
    \end{tabular}
    \caption{Study Scale}
    \label{tab:StudyScaling}
\end{table}

\subsection{Random Index table}
\begin{table}[!htpb]
    \centering
    \scalebox{0.85}
    {
    \begin{tabular}{c|c|c|c|c|c|c|c|c|c|c|c|c|c|c|c}
         Order&1& 2& 3& 4 &5 &6& 7& 8 &9 &10& 11 &12 &13 &14& 15\\
         \hline
         RI &0& 0& 0.52& 0.89 &1.11 &1.25& 1.35 &1.40& 1.45& 1.49 &1.52 &1.54 &1.56 &1.58& 1.59
    \end{tabular}
    }
    \caption{Random Index (RI) Table}
    \label{tab:RandomIndexTable}
\end{table}
\subsection{Derivations}
 The following section will give a derivation of Saaty's consistency ratio. When the right hand side of matrix \ref{Matrix1} is multiplied by $\Vec{w}=\begin{bmatrix} w_1&w_2&\hdots&w_n \end{bmatrix}^T$, we get 
\begin{equation*}\label{Matrix2}
    M\Vec{w}=
    \begin{bmatrix}
    \frac{w_1}{w_1}&\frac{w_1}{w_2}&\hdots&\frac{w_1}{w_n}\\
    \frac{w_2}{w_1}&\frac{w_2}{w_2}&\hdots&\frac{w_2}{w_n}\\
    \vdots&\vdots&\ddots&\vdots\\
    \frac{w_n}{w_1}&\frac{w_n}{w_2}&\hdots&\frac{w_n}{w_n}\\
\end{bmatrix}
    \begin{bmatrix}
        w_1\\w_2\\\vdots\\w_n 
    \end{bmatrix}
    =
    \begin{bmatrix}
        \sum_{i=1}^n w_1\\
        \sum_{i=1}^n w_2\\
        \vdots\\
        \sum_{i=1}^n w_n\\
    \end{bmatrix}
    =
    \begin{bmatrix}
        nw_1\\
        nw_2\\
        \vdots\\
        nw_n\\
    \end{bmatrix}
    =n \Vec{w}
\end{equation*}

By definition, if $n$ is an eigenvalue of matrix $M$, then $\Vec{w}$ is an eigenvector of $M$.

\begin{theorem}
    Each row of matrix $M\Vec{w}$ is a constant multiple of a given row
\end{theorem}

\begin{proof}

For the sake of a concise notation, let $M_{ij}=\frac{w_i}{w_j} \ \forall i,j=1,...n$. Recall that  $M_{ij}=M_{ik} M_{kj} \ \forall i,j,k=1,...,n$. Without loss of generality, let $k=1$. Then, $M_{ij}=M_{i1}M_{1j}$, showing that every entry of M can be expressed as constant multiple of the first row. That is,
$$
 \begin{bmatrix}
    M_{11}&M_{12}&\hdots&M_{1n}\\
    M_{21}&M_{22}&\hdots&M_{2n}\\
    \vdots&\vdots&\ddots&&\vdots\\
    M_{n1}&M_{n2}&\hdots&M_{nn}\\
\end{bmatrix}
=
\begin{bmatrix}
    1&1&\hdots&1\\
    M_{21}&M_{21}&\hdots&M_{21}\\
    \vdots&\vdots&\ddots&\vdots\\
    M_{n1}&M_{n1}&\hdots&M_{n1}\\
\end{bmatrix}
\begin{bmatrix}
    M_{1 1}\\M_{1 2}\\\vdots\\M_{1 n}
\end{bmatrix}
$$

This shows that every row of a perfectly consistent pair-wise comparison matrix $M$ is a constant multiple of any row. In this case, row 1.

\end{proof}

This implies that matrix $M$ has rank 1 and therefore has only one non-zero eigenvalue and a corresponding non-trivial eigenvector $\Vec{w}$. The entries of this eigenvector $\Vec{w} = \begin{bmatrix} w_1&w_2&\hdots &w_n \end{bmatrix}^T$, then, is considered the weights corresponding to each choice $C_i$. Saaty then notes that the comparison matrix is consistent if and only if $\lambda_{max}=n$, a simple and elegent proof of which can be found in Ed Barbeau's article \citep{Barbeau}.




\end{document}